\documentclass[aap,MSNbibl,seceqn,dvips]{arximspdf}
\usepackage{graphicx}
\doi{10.1214/13-AAP937} 
\volume{24}
\issue{2}
\pubyear{2014}
\firstpage{857}
\lastpage{897}

\makeatletter
\newcommand{\rrvert}{\vert}
\newcommand{\llvert}{\vert}
\newtheorem{theo}{Theorem}[section]
\newproclaim{defin}[theo]{Definition}
\newtheorem{Lem}[theo]{Lemma}
\newproclaim{rmk}[theo]{Remark}
\newtheorem{cor}[theo]{Corollary}
\makeatother

\begin{document}
\begin{frontmatter}

\title{The \textit{K}-process on a tree as a scaling limit\\ of the
GREM-like trap model}
\runtitle{\textit{K}-process on a tree}

\begin{aug}
\author[A]{\fnms{L.~R.~G.} \snm{Fontes}\corref{}\ead[label=e1]{lrenato@ime.usp.br}\thanksref{t1}},
\author[B]{\fnms{R.~J.} \snm{Gava}\ead[label=e2]{gavamat@yahoo.com.br}\thanksref{t2}}
\and
\author[C]{\fnms{V.} \snm{Gayrard}\ead[label=e3]{veronique@gayrard.net}}
\runauthor{L.~R.~G. Fontes, R.~J. Gava and V. Gayrard}
\affiliation{University of S\~ao Paulo, Federal University of S\~ao Carlos\\ and Aix Marseille Universit\'e}
\address[A]{L.~R.~G. Fontes\\
IME-USP\\
Rua do Mat\~ao 1010\\
Cidade Universit\'aria\\
05508-090, S\~ao Paulo SP\\
Brazil\\
\printead{e1}} 
\address[B]{R.~J. Gava\\
CCET-UFSCar\\
Rod.~Washington Luiz\\
km 235, 13565-905, S\~ao Carlos SP\\
Brazil\\
\printead{e2}}
\address[C]{V. Gayrard\\
CNRS, CMI, LAPT\\
Aix Marseille Universit\'e\\
39 rue F.~Joliot Curie 13453\\
Marseille cedex 13\\
France\\
\printead{e3}}
\end{aug}
\thankstext{t1}{Supported in part by CNPq Grant 305760/2010-6 and
FAPESP Grant 2009/52379-8.}
\thankstext{t2}{Supported by FAPESP fellowship 2008/00999-0.}

\received{\smonth{9} \syear{2012}}
\revised{\smonth{4} \syear{2013}}

%
\begin{abstract}
We introduce trap models on a finite volume $k$-level tree as a class
of Markov jump processes with state space the leaves of that tree. They
serve to describe the GREM-like trap model of Sasaki and Nemoto. Under
suitable conditions on the parameters of the trap model, we establish
its infinite volume limit, given by what we call a $K$-process in an
infinite $k$-level tree. From this we deduce that the $K$-process also
is the scaling limit of the \mbox{GREM-}like trap model on extreme time scales
under a fine tuning assumption on the volumes.
\end{abstract}

%

\end{frontmatter}

\section{Introduction}
\label{intro}

The long time behavior of slow dynamics in random environments and
phenomena like \textit{aging} 
is a research theme of recent interest. Trap models and related
stochastic processes have been proposed as simple models where these
issues can be studied and understood on a rigorous basis. Perhaps the
simplest such models are Markov jump processes on given graphs with
simple symmetric random walks as embedded chains. The mean jump times
at the vertices are random i.i.d. parameters, with
heavy tailed distribution, 
that may be seen as the depths of traps, playing the role of the random
environment. The case of ${\mathbb Z}^d$ was extensively analyzed in
the physics~\cite{knne,kncb} as well as mathematical
literature~\cite{knfin,knacm,knac,knbc}. The case of the
\textit{complete graph} was introduced in~\cite{knbd} as a toy model
for the aging behavior of the REM, and is well
understood~\cite{knbd,knbf,knfm,kng,knbfggm}. The actual REM dynamics
(with Gibbs factors instead of the i.i.d. heavy tailed random
variables) was studied in~\cite{knabg1,knabg2}, where it is shown that
aging is the same as in the complete graph. Refined understanding of
this dynamics on a wide range of time scales\vadjust{\goodbreak} was obtained in
\cite{knabg1,knabg2,knac1,knfl,kng2,knac2}. A natural next step to the
analyzes of the REM is to consider correlated Hamiltonians, namely the
$p$-spin SK models and the GREM.
The $p$-spin dynamics was studied in~\cite{knbbc,knbg,knbgs} in a
particular range of time scales and temperature parameters where aging
is the same as in the REM. At present, however, there is no rigorous
results about the GREM dynamics. The only results available are
nonrigorous theoretical results~\cite{knbd,knsn1,knsn2} and concern
trap-like models of this dynamics. In this work we consider one of
these models, namely the GREM-like trap
model introduced by Sasaki and Nemoto~\cite{knsn1}. 

%
\begin{figure}[b]

\includegraphics{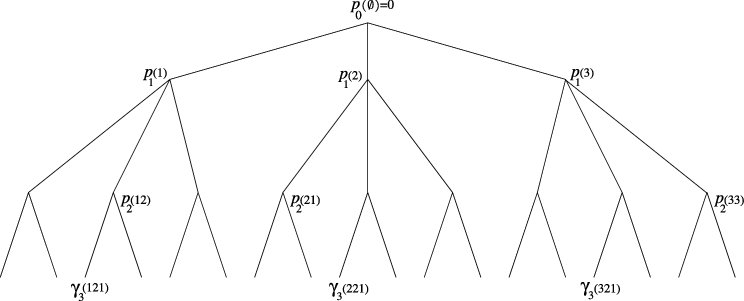}

\caption{A representation of\, ${\mathbb T}^F_3$ with $M_1=M_2=3$ and
$M_3=2$, with coin tossing and waiting time parameters
appearing beside or below a few vertices, respectively.}
\label{figtree}
\end{figure}

Let us first describe the model with a fixed deterministic environment,
which we call the \textit{trap model on a tree}, and come back to the
GREM-like trap model after that. Let $M_1,\ldots,M_k$ be positive
integers, and consider a $k$-level rooted tree whose first generation
has size $M_1$, and such that each vertex in generation $j-1$ has
$M_j$~offspring at generation $j$, $j=2,\ldots,k$. The state space of
our model are the leaves of that tree. The parameters of the model are
as follows. To each leaf vertex we will attach a positive parameter
$\gamma_k$, dependent on the vertex. To interior vertices (not leaves),
we attach probabilities $p_j$, $j=1,\ldots, k-1$, also dependent on the
vertex, all of them positive, except for $p_{\mathrm{root}}$, which
vanishes. See Figure~\ref{figtree} below. The transition mechanism of
the trap model is as follows. Once in a given leaf vertex $x$ of the
tree, the process waits an exponential time with mean $\gamma_k(x)$ and
then jumps. The destination of the jump, another leaf vertex of the
tree---let us call it $y$---is chosen as follows. An ancestor of $x$ on
the tree is first chosen by going up the path from $x$ to the root, and
independently flipping coins whose probabilities of heads are the
$p_j$'s encountered along the way, until tails come up for the first
time. The corresponding stopping vertex is the chosen ancestor; let us
call it $z$. We then choose $y$ uniformly at random among the leaf
vertices descending from $z$. The trap model on a tree is thus fully
described.

The GREM-like trap model is the trap model on a ($k$-level) tree for
which the $\gamma_k$'s as well as the inverse of the $p_j$'s,
$j=1,\ldots, k-1$, are random variables, independent over the vertices,
whose common distribution on a given level $j$ is in the basin of
attraction of a stable distribution with index $\alpha_j$ such that
$0<\alpha_j<1$, $j=1,\ldots, k$. More detailed descriptions of both
the trap model on a tree and the GREM-like trap model are done in
Section~\ref{tree}.

We are interested in the long time behavior of the GREM-like trap model
as the volume diverges. There is of course an issue of how time scales
with the volumes, and how the volumes $M_j$ and the indices $\alpha
_j$, $j=1,\ldots, k$, relate to each other. In this paper we derive a
scaling limit for the process at times of the order of the maxima of
the $\gamma_k$'s---we qualify this time scale as \textit{extreme}. The
$\alpha_j$'s will be taken in strictly increasing order. The volumes
will be related to each other in what we call the \textit{fine tuning}
regime [see~(\ref{eqft}) on Section~\ref{GREM}]. On the extreme time
scale the process is close to equilibrium, so aging does not take place
in this regime. In our setting, aging requires taking a second limit,
after first sending the volume to infinity: the macroscopic time must
then be sent to zero (as discussed, e.g., in~\cite{kng1}); this will be
done in a follow-up paper. Alternatively, we may take a single limit,
with a smaller time scale than the extreme one---this is done
in~\cite{kngg}; see Remark~\ref{rmksn} below.

A scaling limit for the GREM-like trap model is stated and proved in
Section~\ref{convergence}, under the conditions outlined above; see
Theorem~\ref{scal}. In the same section, we state and prove a general
infinite volume limit result for the trap model on a tree; see
Theorem~\ref{GK}. The proof of Theorem~\ref{scal} is obtained in
Section~\ref{GREM} by verifying the conditions of Theorem~\ref{GK}.

In order to perform the infinite volume limit of the trap model on a
tree, we consider an alternative description of that dynamics, since
the original description does not straightforwardly suggest an infinite
volume version. This is done in Section~\ref{finitevolume}. This
representation, a key
element of the paper, 
immediately suggests an infinite volume limit version of the finite
volume dynamics, introduced
in Section~\ref{GKprocess}.

\section{The model}\label{tree}

We describe the trap model on a tree in detail now. Let us start with
the tree. Throughout, $k$ will be a fixed integer in $\mathbb{N}_{*}:=
\{ 1,2,\ldots\}$. Consider~$k$ numbers $M_{1},\ldots,M_{k}
\in\mathbb{N}_{*}$, sometimes below called volumes, and let
$\mathcal{M}_{j} = \{1, \ldots, M_j\}$,
$\mathcal{M}|_{j}=\mathcal{M}_1\times\cdots\times\mathcal{M}_j$,
$j=1,\ldots,k$. Let us then consider the tree rooted at $\varnothing$
\begin{equation}
\label{eqtkf} {\mathbb T}^F_k=\bigcup
_{j=0}^k\mathcal{M}|_{j},
\end{equation}
where $\mathcal{M}|_{0}=\{\varnothing\}$. We will use the notation
$x|_j\equiv(x_1,\ldots, x_j)$ for a generic element of
$\mathcal{M}|_j$. We will also use the notation
\begin{equation}
\label{eqbm} \mathcal{M}|^{j}=\mathcal{M}_j\times\cdots
\times\mathcal{M}_k,
\end{equation}
$1\leq j\leq k$, and let $x|^j\equiv(x_j,\ldots, x_k)$ denote a
generic element of $\mathcal{M}|^{j}$.\vadjust{\goodbreak}

${\mathbb T}^F_k$ is of course a finite tree, and this is emphasized in
the notation by the use of the superscript ``$F$.''
We understand the root to be at the $0$th generation of ${\mathbb
T}^F_k$, and $x|_j\in \mathcal{M}|_{j}$ to be in its $j$th generation,
$1\leq j\leq k$. We will sometimes use simply $x$ for $x|_k$. Given
$0\leq i<j\leq k$, $x|_i\in\mathcal{M}|_i$ and $x'|_j\in
\mathcal{M}|_j$, we will regard $x|_i$ as an \textit{ancestor} of
$x'|_j$ whenever $x|_i=x'|_i$.

Let $\bar{\mathbb N}_{*}= {\mathbb N}_{*}\cup\{\infty\}$. In the
sections below, we will consider the infinite tree
\begin{equation}
\label{eqtk} {\mathbb T}_k=\bigcup_{j=0}^k
\bar{\mathbb N}_{*}^j,
\end{equation}
where $\bar{\mathbb N}_{*}^0=\{\varnothing\}$,
with generations and ancestors as in ${\mathbb T}^F_k$.

The dynamics we will consider is a continuous time Markov jump process
on the set of leaves of
${\mathbb T}^F_k$, namely its $k$th generation $\mathcal{M}|_{k}$.

Let us describe the transition mechanism of the process. There will be
a set of parameters for that. In order to distinguish this finite tree
description from the later infinite tree one (to be presented in
Section~\ref{GKprocess} below) on the one hand, and to emphasize the
analogy between the two cases on the other hand, we continue resorting
to the use the superscript ``$F$'' for the set of parameters of the
finite volume process as well.

For $j = 1,\ldots,k-1$, let
\begin{equation}
\label{eqpjf} p^F_j\dvtx  \mathcal{{M}}|_{j}
\rightarrow(0,1)
\end{equation}
and
\begin{equation}
\label{eqgkf} \gamma^F_k\dvtx  \mathcal{{M}}|_{k}
\rightarrow(0,\infty).
\end{equation}
For $x \in\mathcal{{M}}|_{k}$, let $g_{x} \in\{0,1, \ldots,k-1\}$
be a random
variable such that
\begin{equation}
\label{eta} P(g_{x} = i) = \bigl[1 - p^F_{i}(x|_{i})
\bigr] \prod_{j = i + 1}^{k - 1}p^F_{j}(x|_{j}),
\end{equation}
where by convention $\Pi_{j=k}^{k-1}p^F_{j}(x|_j) = 1$ and $p^F_{0}
\equiv0$.

Let $Z^F_{k}$ be a continuous time Markov chain on $\mathcal
{{M}}|_{k}$ as follows. When $Z^F_{k}$ is at $x\in\mathcal{M}|_k$, it
waits an exponential time of mean $\gamma^F_{k}(x)$ and then jumps as
follows. It first looks at a copy $g'_{x}$ of $g_{x}$ (at each time
independent of the copies looked at previously). If $g'_{x} = j$, then,
letting $a_x(j)$ denote the (only) ancestor of $x$ on generation $j$ of
${\mathbb T}^F_k$ [namely $a_x(j)=x|_{j}$], $Z^F_k$ jumps uniformly at
random to one of the \textit{descendants} of $a_x(j)$ in
$\mathcal{{M}}|_{k}$. In other words, given that $g'_{x} = j$, then the
coordinates $x|_{j}(=a_x(j))$ of $x$ are left unchanged, and the
remainder coordinates are chosen uniformly at random on
$\mathcal{M}|^{j+1}$. We may then say that the transition distribution
of the jump chain of $Z^F_k$ from $x$ is the uniform distribution on
the descendants of an ancestor of $x$ whose generation is randomly
chosen according to the distribution of $g_x$.

%
\begin{rmk}\label{coins}
We may understand the random variable $g_x$ as follows. Let us attach
coins to the sites of the tree that are not leaves, namely, the points
of $\bigcup_{j=0}^{k-1}\mathcal{M}|_{j}={\mathbb
T}^F_k\setminus\mathcal{M}|_{k}$, in such a way that the probability of
heads of the coin at site $x|_{j} \in\mathcal{{M}}|_{j}$ is
$p^F_{j}(x|_{j}), j = 1,\ldots,k-1$; see Figure~\ref{figtree}. The coin
of the root has probability $p^F_{0} = 0$ of turning up heads. When it
decides to jump from site $x \in\mathcal{{M}}|_{k}$, $Z^F_k$ first
flips successively the coins of $x|_{k-1}, \ldots, x|_{1}$,
$\varnothing$ (in that order) until it gets tails for the first time,
and then it stops at the respective site. Notice that this procedure is
almost surely well defined since $p^F_{0}=0$. Given that $x|_{j}$
was the stopping site of the procedure,
then $g_x=j$.\looseness=1
\end{rmk}


%
\begin{defin}
We call $Z^F_{k}$ a trap model on ${\mathbb T}^F_k$, or $k$-level trap
model, with \textit{waiting time} parameter $\gamma^F_{k}$ and
\textit{activation} parameters $(p^F_{j})_{j=1}^{k-1}$, and write
\[
Z^F_{k} \sim TM\bigl({\mathbb T}^F_k;
\gamma^F_{k}; \bigl(p^F_{j}
\bigr)_{j=1}^{k-1}\bigr).
\]
\end{defin}

Our main motivation in considering this model is in the particular case
where the parameters are related to the following random variables. For
$j = 1,\ldots,k$, let $\tau_{j}:= \{ \tau_{j}(x|_j); x|_j \in
\mathcal{M}|_{j}\}$ be an i.i.d. family of positive random variables in
the domain of attraction of an $\alpha_{j}$-stable law. Now consider
the $k$-level trap model, with~\textit{waiting time} parameters
$\gamma^F_{k}(x|_k)\equiv\tau_{k}(x|_k)$ and \textit{activation}
parameters $p^F_{j}(x|_j)\equiv1/(1+\tau_{j}(x|_j)), j=1,\ldots,k-1$.
We call this model the \textit{GREM-like trap model} on ${\mathbb
T}^F_k$
with parameters $\tau_{j}, j=1,\ldots,k$. 
We state a scaling limit result for this model in
Section~\ref{convergence} below. In the next two sections we present
supporting material for that result, as anticipated at the end of the
\hyperref[intro]{Introduction}.

\section{A representation of the $k$-level trap model}
\label{finitevolume}

In this section we will inductively construct a process $X^F_{k}$ on
$\mathcal{{M}}|_{k}$, under a particular choice of whose parameters it
is a version of the $k$-level trap model of last section. As explained
in the \hyperref[intro]{Introduction}, this particular version will
help us to formulate the infinite volume limit of the latter model.

The process of this section will involve a set of parameters
$\gamma^F_j\dvtx  \mathcal{{M}}|_{j} \rightarrow\mathbb(0, \infty), j
=1,\ldots,k $, for given $M_{1},\ldots,M_{k} \in\mathbb{N}_{*}$.

In order to have our inductive construction go smoothly, we introduce
an auxiliary process $Y^F_{j}$, for bookkeeping reasons
only,\vspace*{-1pt} as will be explained below. We will then have pairs
$(X^F_{j},Y^F_{j})$, $j=1,\ldots,k$. (The auxiliary process will not be
needed in the infinite volume version of $X^F_{k}$ to be introduced in
Section~\ref{GKprocess}.) We first define the process $(X^F_1,Y^F_1)$.
$X^F_1$ is a continuous time Markov chain on
$\mathcal{{M}}|_{1}(=\mathcal{{M}}_{1})$ that, when at
$x_1\in\mathcal{{M}}|_{1}$, waits an exponential time of mean $\gamma
^F_1(x_1)$ and then jumps uniformly to a site in $\mathcal{{M}}|_{1}$.
We will construct $X^F_{1}$ in the following way.

Let $\mathcal{N}_{1}=\{(N_r^{(x_1,1)})_{r\geq0}, x_1 \in\mathbb{N}_*\}$
be i.i.d. Poisson processes of rate 1, and let $\sigma_{i}^{x_{1},1}$
be the $i$th mark of $N^{(x_1,1)}$ (viewed as a point process), $i
\geq1$. We will call $\mathcal{S}^F_{1} = \{\sigma_{i}^{(x_{1},1)};
x_{1} \in\mathcal{M}_{1}, i \geq1 \}$ the set of \textit{marks of the
first level} of $X^F_{k}$. Let $\mathcal{T}_{1}=\{T_{s}^{(1)}, s
\in\mathbb{R}^{+}:= [0, \infty) \}$ be i.i.d. exponential random
variables of rate 1. $\mathcal{N}_{1}$~and~$\mathcal{T}_{1}$ are
assumed independent.

For $s \in\mathcal{S}^F_{1}$, let $\xi^F_{1}(s) = x_1$ if $s =
\sigma_{j}^{x_1,1}$ for some $x_{1} \in\mathcal{M}_{1}$ and
$j \geq1$. 
Notice that $\xi^F_{1}$ is well defined almost surely. Let us now
define a measure $\mu^F_{1}$ on $\mathbb{R}^{+}$ as follows: $\mu^F_1
(\{ s \}) = \gamma^F_{1}(\xi^F_1(s)) T_{s}^{(1)}$ if $s
\in\mathcal{S}^F_{1}$ and $\mu^F_1 (\mathbb{R}^{+}
\setminus\mathcal{S}^F_{1}) = 0$.

%
\begin{rmk}
We note that $\xi^F_1(s), s \in\mathcal{S}^F_{1}$, are
i.i.d. uniform random variables in $\mathcal{M}_{1}$.
\end{rmk}

For $r \geq0$, let
\begin{equation}
\Gamma^F_{1}(r):= \mu^F_1
\bigl([0,r]\bigr).
\end{equation}
For $t \geq0$, let
\begin{equation}
\varphi^F_{1}(t):= \bigl(\Gamma^F_{1}
\bigr)^{-1}(t) = \inf\bigl\{ r \geq0; \Gamma^F_{1}(r)
> t \bigr\}
\end{equation}
be the (right continuous) inverse of $\Gamma^F_{1}$.

Let us recall that $\bar{\mathbb{N}}_{*}=\{1,2,\ldots,\infty\}$. We
define the process $(X^F_{1}, Y^F_{1})$ on $(\bar
{\mathbb{N}}_{*},\mathbb{R}^{+})$ as follows. For $t \geq0$,
\begin{equation}
\bigl(X^F_1, Y^F_1\bigr)
(t)= \bigl(\xi^F_1\bigl(\varphi^F_{1}(t)
\bigr), \varphi^F_{1}(t)\bigr).
\end{equation}
Let us suppose
$(X^F_{j},Y^F_{j})$ is defined for $j = 1,\ldots,l-1$, $l\leq k$.

%
\begin{defin}\label{citrap}
An interval $I \subset\mathbb{R}^{+}$ is a constancy interval of
$(X^F_{j},Y^F_{j})$
if $(X^F_{j},Y^F_{j})$ is constant over $I$, that is,
\begin{equation}
\bigl(X^F_{j}, Y^F_{j}\bigr) (r) =
\bigl(X^F_{j},Y^F_{j}\bigr) (s)\qquad\mbox{for all } r,s \in I
\end{equation}
and $I$ is maximal with that property.
\end{defin}

The maximality condition and right continuity of $(X^F_{j},Y^F_{j})$
implies that
$I = [a,b)$ for some $0 \leq a < b$. We are now ready to define
$(X^F_{l}, Y^F_{l})$
for $2\leq l \leq k$.

Let $\mathcal{I}^F_{l-1}$ be the collection of constancy intervals of
$(X^F_{l-1}, Y^F_{l-1})$. Let also
$\mathcal{N}_{l}=\{(N_r^{(x_l,l)})_{r\geq0}, x_l \in\mathbb{N}_* \}$ be
i.i.d. Poisson processes of rate 1. Let $\sigma_{i}^{x_{l},l}$ the
$i$th mark of $N^{(x_l,l)}, i \geq1$. We will call
$\mathcal{S}^F_{l}=\{\sigma_{i}^{(x_{l},l)}; x_{l}\in\mathcal{M}_{l},
i\geq1\}$ the set of Poisson marks of the $l$th level, and
$\mathcal{R}^F_{l}=\{a; I=[a,b)$ and $I\in\mathcal{I}^F_{l-1}\} $ the
set of extra marks of the $l$th level. Notice that $\mathcal{R}^F_{l}$
is the set of left endpoints of intervals of $\mathcal{I}_{l-1}$. We
call $\mathcal{S}^F_{l} \cup\mathcal{R}^F_{l}$ the set of marks of the
$l$th level; see Figure~\ref{figconst}. Let
$\mathcal{T}_{l}=\{T_{s}^{(l)}, s \in\mathbb{R}^{+}\}$ be i.i.d.
exponential random variables of rate 1. $\mathcal{N}_{l}$ and
$\mathcal{T}_{l}$ are assumed independent and are independent of
$\mathcal{N}_j$ and $\mathcal{T}_{j}$ for $j < l$.

%
\begin{figure}

\includegraphics{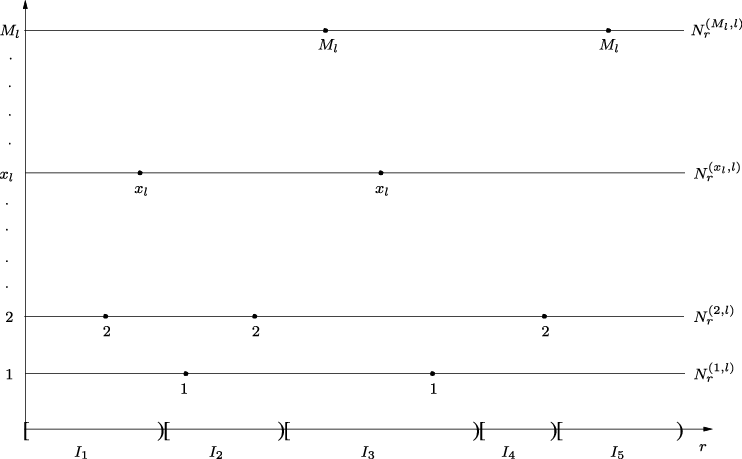}

\caption{Representation of the timelines of the Poisson point
processes entering the definition of~$X^F_{l}$, one for each $x\in
\mathcal{M}_l$.
Another ingredient are the constancy intervals of~$X^F_{l-1}$, here
successively represented in the $x$-axis as $I_1,\ldots,I_5$.}\label{figconst}
\end{figure}

To each $s \in\mathcal{R}^F_{l}$, we associate a uniform random
variable $U_{l}(s)$
on $\{ 1,\ldots, M_{l} \}$. Assume that $\{ U_{l}(s), s \in\mathcal
{R}^F_{l}, l \geq1 \}$
are mutually independent and independent of the other random variables
in the model.
Let
\begin{equation}
\label{xi} \xi^F_{l}(s) = \cases{ x_l, &
\quad if $s=\sigma_{j}^{(x_l,l)}$ for some $x_{l} \in
\mathcal{M}_{l}$ and $j \geq1$,
\cr
U_{l}(s), &\quad if $s
\in\mathcal{R}^F_{l}$.}
\end{equation}

We will call $\xi^F_{l}(s)$ the \textit{label} of $s\in\mathcal
{R}^F_{l}\cup\mathcal{S}^F_{l}$, $1\leq l\leq k$ (where
$\mathcal{R}^F_{1}=\varnothing$).

Notice that in the first case of~(\ref{xi}) above $s \in\mathcal
{S}^F_{l}$, and that $\xi^F_{l}$ is well defined almost surely. Let us
now define a measure $\mu^F_{l}$ on $\mathbb{R}^{+}$ as follows:
\begin{equation}
\mu^F_l\bigl(\{s\}\bigr)=\gamma^F_{l}
\bigl(X^F_{l-1}(s),\xi^F_{l}(s)
\bigr)T_{s}^{(l)}\qquad\mbox{if }s\in\mathcal{S}^F_{l}
\cup\mathcal{R}^F_{l}
\end{equation}
and $\mu^F_l (\mathbb{R}^{+} \setminus(\mathcal{S}^F_{l} \cup
\mathcal{R}^F_{l})) = 0 $. Notice that $\mathcal{S}^F_{l}
\cap\mathcal{R}^F_{l} = \varnothing$ almost surely.

%
\begin{rmk}
We note that $\xi^F_j(s), s \in\mathcal{S}^F_{j}\cup\mathcal
{R}^F_{j}$, are i.i.d. uniform random variables in
$\mathcal{M}_{j}, j = 1,\ldots,l$.
\end{rmk}

For $r \geq0$, let
\begin{equation}
\Gamma^F_{l}(r):= \mu^F_l
\bigl([0,r]\bigr).
\end{equation}
For $t \geq0$, let
\begin{equation}
\varphi^F_{l}(t):= \bigl(\Gamma^F_{l}
\bigr)^{-1}(t) = \inf\bigl\{ r \geq0\dvtx  \Gamma^F_{l}(r)
> t \bigr\}
\end{equation}
be the inverse of $\Gamma^F_{l}$.

We define the process $(X^F_{l},Y^F_{l})$ on $(\mathcal{{M}}|_{l},
\mathbb{R}_{+}^{l})$ as follows. For $t \geq0$,
\begin{equation}
\bigl(X^F_l, Y^F_l\bigr) (t) =
\bigl( \bigl(X^F_{l-1}\bigl(\varphi^F_{l}(t)
\bigr),\xi^F_{l}\bigl(\varphi^F_{l}(t)
\bigr) \bigr); \bigl(Y^F_{l-1}\bigl(\varphi^F_{l}(t)
\bigr), \varphi^F_{l}(t) \bigr) \bigr);
\end{equation}
see Figure~\ref{figclock}.

%
\begin{figure}

\includegraphics{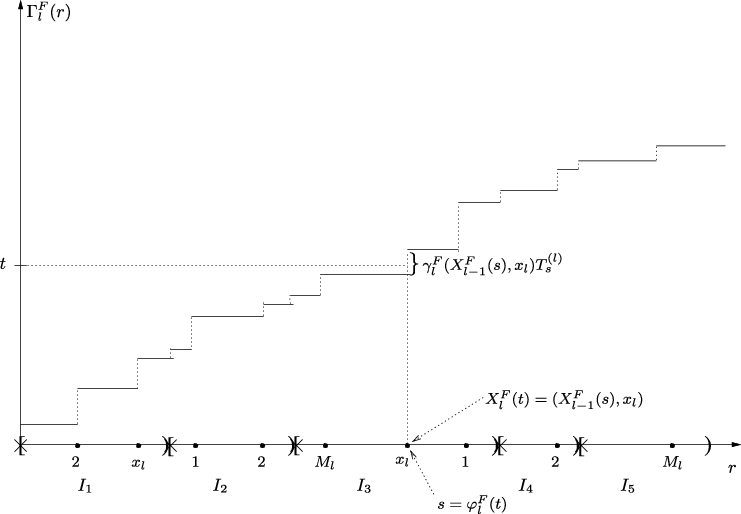}

\caption{Illustration of the construction of $X^F_{k}$ from the
constancy intervals $I_1=[a_1,b_1),\ldots,$ $I_5=[a_5,b_5)$ of
$X^F_{k-1}$ on $x$-axis, Poisson
marks and extra marks [here represented as crosses on left endpoints of
the constancy intervals, with (not shown) labels $U_l(a_1),\ldots,U_l(a_5)$,
resp.]. The Poissonian points may be seen as projected down from
timelines of Figure~\protect\ref{figconst}. The point ``$s$'' in the $x$-axis
equals $\varphi^F_l(t)$.}\label{figclock}
\end{figure}

%
\begin{rmk}\label{prop}
We note that each interval $I$ of $\mathcal{I}^F_{j}, j \geq2$ can be
identified with a jump of $\Gamma^F_{j}$, that is, $\mathcal{I}^F_{j} =
\{ [\Gamma^F_{j}(r-), \Gamma^F_{j}(r))\ne\varnothing; r \geq0 \} = \{
I^F_j(s):= [\Gamma^F_{j}(s-), \Gamma^F_{j}(s)); s \in
\mathcal{S}^F_{j}\cup\mathcal{R}^F_{j}\}$, and the lengths of the
intervals of $\mathcal{I}^F_{j}$, namely $\{ |I^F_j(s)|, s
\in\mathcal{S}^F_{j} \cup\mathcal{R}^F_{j}\}$, are independent
exponential random variables with means $\{ \gamma^F_{j}(X^F_{j-1}(s),
\xi^F_{j}(s)), s \in\mathcal{S}^F_{j}\cup\mathcal{R}^F_{j} \}$,
respectively.
\end{rmk}

At this point we may observe that our interest is in $X^F_{k}$; as
anticipated above, $Y^F_{k}$ is introduced for bookkeeping purposes,
solely for the convenience of having the property mentioned in
Remark~\ref{prop}. In Section~\ref{GKprocess} below we will introduce
an infinite volume version of $X^F_{k}$ for which the respective
version of $Y^F_{k}$ will not be needed explicitly, and thus not
explicitly introduced.

%
\begin{rmk}\label{geonumber}
We\vspace*{-1pt} note that the number of marks of $\mathcal{S}^F_{j}$
in each interval $ I \in\mathcal{I}^F_{j-1}$ (when integrated with
respect to the exponential interval length; see Remark~\ref{prop}
above) is a geometric random variable\footnote{In this paper, we call a
geometric random variable one whose probability function is given by
$p^n(1-p)$, $n=0,1,\ldots,$ where $p$ is a parameter in $(0,1)$, and
whose mean is thus in terms of $p$ given by $(1-p)/p$.} with mean
$M_{j}\gamma^F_{j-1}(X^F_{j-1}(s))$, where $s \in\mathcal{S}^F_{j-1}
\cup\mathcal{R}^F_{j-1}$ is such that $I = I^F_{j-1}(s)$, and that each
such interval has exactly one mark of $\mathcal{R}^F_{j}$ at its left
end. So, the total number of marks (of $\mathcal{S}^F_{j}\cup
\mathcal{R}^F_{j}$) within $I$ is the above mentioned geometric
variable plus one.
\end{rmk}

%
\begin{defin}\label{klt}
We call $X^F_{k}$ defined above the trap model on ${\mathbb T}^F_k$, or
\mbox{$k$-}level trap model, with parameter set
$\underline{\gamma^F_{k}} = \{ \gamma^F_{i}; i = 1,\ldots,k \}$.
Notation: $X^F_{k} \sim TM({\mathbb T}^F_k; \underline{\gamma^F_{k}})$.
\end{defin}

We now make the connection between the models of Sections~\ref{tree}
and~\ref{finitevolume}, a key result of this paper, which in particular
establishes that the latter is a representation of the former under the
appropriate relationship of their respective set of parameters, thus
justifying the common terminology.
%
\begin{Lem}\label{lmrep}
Let $X^F_{k}$ be as above and $Z^F_{k}$ be as in Section~\ref{tree},
that is,
\[
Z^F_{k} \sim TM\bigl(\mathcal{{M}}|_{k};
\gamma^F_{k};\bigl(p^F_{j}
\bigr)_{j=1}^{k-1}\bigr).
\]
Suppose
\begin{equation}
\label{p} p^F_{j}(x|_{j}):=
\frac{1}{1 + M_{j+1}\gamma^F_{j}(x|_{j})}
\end{equation}
for all $x|_{j} \in\mathcal{{M}}|_{j}$ and $j =1,\ldots,k-1$.
Then $X^F_{k}$ and $Z^F_{k}$ have the same distribution.
\end{Lem}

\begin{pf}
We begin with the following remark concerning $Z^F_k$, which follows
immediately from the construction of that process in
Section~\ref{tree}.

%
\begin{rmk} \label{ze}
Let $Z^F_{k,i}, i = 1,\ldots,k$, be the $i$th coordinate of $Z^F_{k} =
(Z^F_{k,1},\break \ldots, Z^F_{k,k})$, and let $Z^F_{k}|_j = (Z^F_{k,1},
\ldots,Z^F_{k,j}), j = 1,\ldots,k$. As pointed out in
Section~\ref{tree} above, the jump chain of $Z^F_{k}$, let us call it
$J^F_{k} = (J^F_{k,1},\ldots,J^F_{k,k})$, with $J^F_{k}|_{j} =
(J^F_{k,1}, \ldots,J^F_{k,j})$, $j = 1,\ldots,k$, can be described in
terms of the sucessive flips of the coins of $J^F_{k}|_{k-1}, \ldots,
J^F_{k}|_{1}$; see Remark~\ref{coins}. After $n$ jumps of $J^F_{k}$,
let us consider the event $A_{k-1}(n) = \{\mbox{flip of the coin of
}J^F_{k}|_{k-1}(n)\mbox{ results in}\break \mbox{heads} \}$. In terms of the random
variable $g_{J^F_{k}(n)}$, we have $A_{k-1}(n)=\{
g_{J^F_{k}(n)}<k-1\}$. We now remark that, given $J^F_{k}(n)=x|_k$ and
$A_{k-1}(n)$, the distribution of the jump from $J^F_{k}|_{k-1}(n)$ is
the same as that from $J^F_{k-1}(n)$ given $J^F_{k-1}(n)=x|_{k-1}$.
Since the distribution of a jump from $J^F_{k,k}(n)$ is always uniform
in $\mathcal{M}_{k}$ independent from anything else, we have that,
given $J^F_{k-1}(n)=x|_k$ and $A_{k-1}(n)$, the distribution of the
jump from $J^F_{k}(n)$ is the same as the joint distribution of the
jump from $J^F_{k-1}(n)$ given $J^F_{k-1}(n)=x|_{k-1}$, and an
independent uniform random variable in $\mathcal{M}_{k}$.
\end{rmk}

For $k = 1$, we remark that the transition probabilities of $X^F_{1}$
are always uniform in $\mathcal{M}_{1}$, since the labels of the
successive points of $\mathcal{S}^F_{1}$ have this property and are
independent of each other. So, $X^F_{1}$ and $Z^F_{1}$ are processes
with the same state space and transition probabilities, and the holding
times are clearly matched. The result follows for $k = 1$.

We now proceed by induction on $k$. Let us suppose that the result
holds for $k = K-1$ for some $K \geq2$. To show that $X^F_{K} \sim
Z^F_{K}$, it is enough to identify the transition mechanisms of both
processes. Let us thus fix a time $t\geq 0$ and a point
$x|_{K}\in\mathcal{{M}}|_{K}$. Given that either process is at $x|_K$
at time $t$, then both jump times are exponentially distributed with
mean $\gamma^F_{K}(x|_K)$. So far, we have an identification. Now let
us identify the jump mechanisms of both processes.

Let ${\mathfrak N}_t$ denote the number of jumps of $Z^F_{K}$ up to
time $t$. As discussed in Remark~\ref{ze} above, given
$Z^F_{K}(t)=x|_K$ and $A_{K-1}({\mathfrak N}_t)$ we have that
${Z}^F_{K}|_{K-1}(t)$ jumps as $Z^F_{K-1}(t)$ given
$Z^F_{K-1}(t)=x|_{K-1}$, and $Z^F_{K,K}(t)$ jumps uniformly in
$\mathcal{M}_{K}$, and the jumps of ${Z}^F_{K}|_{K-1}(t)$ and
$Z^F_{K,K}(t)$ are independent. Let us identify a \textit{coin tossing}
mechanism in the jump of $X^F_{K}$. Let $I$ be the interval of
$\mathcal{I}^F_{K-1}$ containing $\varphi_K(t)$, and let
$I'=(\varphi_K(t),\infty)\cap I$. The lack of memory of the exponential
distribution of $I$ (see Remark~\ref{prop} above) implies that, given
that $X^F_{K}(t)=x|_K$, then $|I'|$ is an exponential random variable
with mean $\gamma^F_{K-1}(x|_{K-1})$.
The mechanism in the $X^F_{K}$~process that plays the role of (the
first) coin tossing is whether or not
$I'$ contains at least one (Poisson) mark. Let us call $\bar{A}_{K-1}$
the event that $I'$ contains no mark. This corresponds to the coin at
$x|_{K-1}$ turning up heads. This means that ${X}^F_{K}|_{K-1}$ will
take a jump, which we identify as a jump of $X^F_{K-1}$, independent of
the (uniform in~$\mathcal{M}_{K}$) accompanying jump of $X^F_{K,K}$.

At this point we should stress that a particular element of the
construction of $X^F_{K}$ plays a key role in this identification,
namely, the inclusion of the marks of~$\mathcal{R}^F_{K}$. This
guarantees\vspace*{-1pt} that each interval of constancy $I$ of
$\mathcal{I}^F_{k-1}$ gets at least one mark.
Without these marks, the jump of ${X}^F_{K}|_{K-1}$ 
might not coincide with that of $X^F_{K -1}$---that would happen if
(and only if) the first interval of $\mathcal{I}^F_{K-1}$ neighboring
$I$
to the right had got no mark.

By the induction hypothesis, $X^F_{K-1} \sim Z^F_{K-1} \sim
{Z}^F_{K}|_{K-1}$; clearly $X^F_{K,K}
\sim Z^F_{K,K}$. We close the argument in two steps. The first one is
to show that
\begin{eqnarray}\label{A{k-1}}
P \bigl(\bar{A}_{K-1}\llvert
X^F_{K}(t)=x|_K \bigr) &=& \frac{1}{1+M_{K}\gamma^F_{K-1}(x|_{K-1})}
\nonumber\\[-8pt]\\[-8pt]
&=& P \bigl(A_{K-1}({\mathfrak N}_t)\llvert
Z^F_{K}(t)=x|_K \bigr)\nonumber
\end{eqnarray}
and then use the hypothesis. But it follows from our discussion above
and the construction of $X^F_{K}$ that the left-hand side
of~(\ref{A{k-1}}) equals $P(N'_{T'}=0)$, where $N'$~is a Poisson
process of rate $M_k$ and $T'$ an independent exponential random
variable with mean $\gamma^F_{K-1}(x|_{K-1})$, and a simple computation
yields the result.

And the last step is to argue that, given $\bar{A}_{K-1}^{c}$, then
${X}^F_{K}|_{K-1}$ does not move, and we have only the uniform in
$\mathcal{M}_{K}$ jump
of $X^F_{K,K}$, which agrees with the corresponding move of $Z^F_{K}$
given $A_{K-1}^{c}$.
\end{pf}

%
%
\section{$K$-process on a tree}\label{GKprocess}

$K$-processes on $\bar{\mathbb N}_{*}$ were introduced in~\cite{knfm}
in the study of limits of trap models in the complete graph. They
appear as scaling limits of the REM-like trap model in the complete
graph. Below we introduce an extension of that model to a model on
$\bar{\mathbb N}^k_{*}$, which we will view as the leaves of a tree
with $k$ generations,
as done similarly in the previous sections. 
As anticipated in the \hyperref[intro]{Introduction} and established in
the next section, the process of this section turns up in limit results
for the processes of the previous sections as volume diverges.

Let $\gamma_j\dvtx  \mathbb{N}_{*}^{j} \rightarrow(0,\infty)$,
$j=1,\ldots,k $, be such that, making
\begin{equation}
\label{gamma} \bar{\gamma}_{j}(x|_{j}):=
\gamma_{1}(x|_1)\times\gamma_{2}(x|_2)
\times\cdots\times\gamma_{j}(x|_{j}),
\end{equation}
we have
\begin{equation}
\label{eqfin} \sum_{x|_j\in\mathbb{N}_{*}^{j} }\bar{\gamma}_{j}(x|_{j})<
\infty.
\end{equation}


We will construct a process $X_{k}$ on $\bar{\mathbb{N}}_{*}^{k}$
inductively, similarly as in Section~\ref{finitevolume}. This will be a
c\`adl\`ag process, similar to the ones we dealt with so far. First we
define the process $X_1$. It is a continuous time Markov chain on
$\bar{\mathbb{N}}_{*}$ described as follows.

Let $\mathcal{N}_{1}=\{(N_r^{(x_1,1)})_{r\geq0}, x_1 \in\mathbb
{N}_{*}\}$ be
i.i.d. Poisson processes of rate 1. Let $\sigma_{i}^{x_{1},1}$ be the $i$th
mark of $N^{(x_1,1)}, i \geq1$. We will call
$\mathcal{S}_{1} = \{\sigma_{i}^{(x_{1},1)}; x_{1} \in\mathbb
{N}_{*}, i \geq1 \}$
the set of marks of the first level of $X_k$.
Let $\mathcal{T}_{1}=\{T_{s}^{(1)}, s \in\mathbb{R}^{+} \}$
be i.i.d. exponential random variables of rate 1. $\mathcal{N}_{1}$ and
$\mathcal{T}_{1}$ are assumed independent.

For $s \in\mathcal{S}_{1}$, let $\xi_{1}(s) = x_1$ if $s = \sigma
_{i}^{x_1,1}$
for some $x_1 \in\mathbb{N}_{*}$ and $i \geq1$. Notice that $\xi
_{1}$ is well\vadjust{\goodbreak}
defined almost surely. Let us now define a measure $\mu_{1}$ on
$\mathbb{R}^{+}$
as follows:
\begin{equation}
\mu_1 \bigl(\{ s \}\bigr) = \gamma_{1}\bigl(
\xi_{1}(s)\bigr) T_{s}^{(1)},\qquad\mbox{if }s \in
\mathcal{S}_{1}\quad\mbox{and}\quad\mu_1 \bigl(
\mathbb{R}^{+} \setminus\mathcal{S}_{1}\bigr) = 0.
\end{equation}

For $r \geq0$, let
\begin{equation}
\Gamma_{1}(r):= \mu_1 \bigl([0,r]\bigr)
\end{equation}
and, for $t \geq0$, let
\begin{equation}
\varphi_{1}(t):= \Gamma_{1}^{-1}(t) = \inf
\bigl\{ r \geq0\dvtx  \Gamma_{1}(r) > t \bigr\}
\end{equation}
be the inverse of $\Gamma_{1}$.

%
\begin{rmk}
\label{atom}
Notice that $\mu_1$ is almost surely a purely atomic
measure whose set of atoms, $\mathcal{S}_{1}$, is a.s. countable and
dense in $\mathbb{R}^{+}$. Moreover, from Lemma~\ref{fin} below, it is
a.s. $\sigma$-finite. These properties imply that $\Gamma_{1}\dvtx
\mathbb{R}^{+} \to\mathbb{R}^{+}$ is a.s. strictly increasing and that
its range $\Gamma_{1}(\mathbb{R}^{+})$ is an uncountable set (since it
is the image of an uncountable set, $\mathbb{R}^{+}$, by a 1 to 1 map)
of Lebesgue measure zero. It follows from this and the independence and
continuity of its constituents that any fixed deterministic $r$ is a.s.
a continuity point of $\Gamma_{1}$. It may also be checked that
$\varphi_{1}\dvtx  \mathbb{R}^{+} \to\mathbb{R}^{+} $ is a.s.
continuous.
\end{rmk}

In order to make the processes to be defined below c\`adl\`ag, we need
the following general
definition.

%
\begin{defin}
Given a function $f\dvtx {\mathbb R}^+\to{\mathbb R}^+$ and
$t\in{\mathbb R}$, we say that $f$~is \textit{upper locally constant}
at $t$ if there exists an $\epsilon>0$ such that $f$ is constant
in $[t,t+\epsilon]$; let $\mathrm{ULC}_f$ denote the set
$\{t\in{\mathbb R}^+\dvtx f$ is upper locally constant at~$t\}$.
\end{defin}

We define $X_{1}$ on $\bar{\mathbb{N}}_{*}$ as follows. For $t \geq0$
\begin{equation}
\label{defk} X_1(t) = \cases{ \xi_{1}\bigl(
\varphi_{1}(t)\bigr), &\quad if $\varphi_{1}(t) \in
\mathcal{S}_{1}$ and $t\in \mathrm{ULC}_{\varphi_{1}}$,
\cr
\infty, &\quad otherwise.}
\end{equation}

Suppose $X_{j}$ is defined for $j = 1,\ldots,l-1$, $2\leq l\leq k$. Let
$\mathcal{N}_{l}= \{(N_r^{(x_l,l)})_{r\geq0},\break  x_l \in\mathbb
{N}_* \}$ be i.i.d. Poisson processes of rate 1. Let
$\sigma_{i}^{x_{l},l}$ the $i$th mark of $N^{(x_l,l)}, i \geq1$. We
will call $\mathcal{S}_{l} = \{ \sigma_{i}^{(x_{l},l )}; x_{l}
\in\mathbb {N}_{*}, i \geq1 \}$ the set of Poisson marks of the $l$th
level. Let $\mathcal{T}_{l}=\{T_{s}^{(l)}, s \in\mathbb{R}^{+}\}$ be
i.i.d. exponential random variables of rate 1. $\mathcal{N}_{l}$ and
$\mathcal{T}_{l}$ are assumed independent and are independent of
$\mathcal{N}_j$ and $\mathcal {T}_{j}$ for~$j < l$.

For $s \in\mathcal{S}_{l}$, let $\xi_{l}(s) = x_l$ if $s = \sigma
_{j}^{(x_l,l)}$
for some $x_l \in\mathbb{{N}}_{*}$ and $j \geq1$. Notice that $\xi
_{l}$ is well
defined almost surely. Let us now define a measure $\mu_{l}$ on
$\mathbb{R}^{+}$
as follows:
\[
\mu_l \bigl(\{ s \}\bigr) = \gamma_{l}
\bigl(X_{l-1}(s), \xi_{l}(s)\bigr)T_{s}^{(l)},\qquad
\mbox{if } s \in\mathcal{S}_{l}\quad\mbox{and}\quad \mu_l \bigl(
\mathbb{R}^{+} \setminus\mathcal{S}_{l}\bigr) = 0.
\]

For $r \geq0$, let
\begin{equation}
\Gamma_{l}(r):= \mu_l \bigl([0,r]\bigr)
\end{equation}
and for $t \geq0$, let
\begin{equation}
\varphi_{l}(t):= \Gamma_{l}^{-1}(t) = \inf\bigl
\{ r \geq0; \Gamma_{l}(r) > t \bigr\}
\end{equation}
be the inverse of $\Gamma_{l}$. $\Gamma_{l}$ will sometimes below be
referred to as \textit{the clock}. It may be (and has been, in the
literature) also called \textit{clock process} (in this case, at level
$l$).

%
\begin{rmk}
\label{atomb}
Remark~\ref{atom} holds with ``1'' replaced by ``$l$.''
In particular, the range of $\Gamma_{l}$ has a.s. Lebesgue measure
zero, $l=1,\ldots,k$, and every fixed deterministic $r$ is a.s. a
continuity point of $\Gamma_{l}$.
\end{rmk}

We define the process $X_{l}$ on $\bar{\mathbb{N}}_{*}^{l}$ as follows.
For $t \geq0$, let
\[
X_l(t)=\cases{ \bigl(X_{l-1}\bigl(\varphi_{l}(t)
\bigr),\xi_{l}\bigl(\varphi_{l}(t)\bigr) \bigr),& \quad if
$\varphi_{l}(t) \in\mathcal{S}_{l}$ and $t\in \mathrm{ULC}_{\varphi _{l}}$,
\vspace*{3pt}\cr
\bigl(X_{l-1}\bigl(\varphi_{l}(t)
\bigr), \infty\bigr),& \quad otherwise.}
\]

%
\begin{defin}
We call $X_{k}$ defined just above the $K$-process on ${\mathbb T}_k$,
or $k$-level $K$-process, with parameter set $\underline{\gamma_{k}} =
\{\gamma_{i}; i = 1,\ldots,k \}$. Notation: $X_{k} \sim K({\mathbb
T}_k,\underline{\gamma_{k}})$.
\end{defin}

%
\begin{rmk}\label{ciK}
Since we only have Poissonian marks in the above definition of $X_{k}$,
we did not have the need of the second coordinate $Y_k$, as in finite
volume, nor did we need to explicitly mention constancy intervals.
The latter notion is nonetheless useful in this context (it will come
up later, in one of our proofs of convergence), and is defined as
follows. Given $1\leq j\leq k$, an interval $I \subset\mathbb{R}^{+}$
is a~constancy interval of $X_{j}$ if it has positive length
\begin{equation}
X_{j}(r) = X_{j}(s)\qquad\mbox{for all } r,s \in I\mbox{ and }I \mbox{ is maximal.}
\end{equation}
The maximality condition and right continuity of $X_{j}$ implies that $I=[a,b)$
for some $0 \leq a < b$. 
\end{rmk}


Pictures like those in Figures~\ref{figconst} and~\ref{figclock} might
be drawn (or perhaps, more accurately, envisioned) for $X_{k}$, with
minor changes: in the present case we would have an infinite sequence
of time lines for the Poisson processes $(N_r^{(x_l,l)})_{r\geq0}$,
$x_l\in\mathbb{N}_*$, in Figure~\ref{figconst}. In
Figure~\ref{figclock}, superscripts ``$F$'' should be dropped
throughout; there would be no extra marks, and thus no crosses; the
Poissonian marks would form a dense set of the $x$-axis; if one wanted
to represent them, the constancy intervals would be such that there
would be an infinite number of them in the neighborhood of any fixed
one of them---or, more precisely, between any two distinct such
intervals, there is a distinct such interval; the graph of $\Gamma_l$
would be that of a strictly increasing function with a dense set of
jumps (the Poissonian marks).\vadjust{\goodbreak}

The next result makes the above construction a.s. well-defined for all
times, and implies that $X_k$ is never absorbed at any state. For its
proof, let us introduce the notation
\[
X_{k} = (X_{k,1},\ldots,X_{k,k}),\qquad k \geq1,
\]
making the coordinates of $X_{k}$ explicit.

%
\begin{Lem}\label{fin}
We have that almost surely $\Gamma_{k}(r)<\infty$ for all $r\in[0,\infty)$ and
$\lim_{r\to\infty}\Gamma_{k}(r)=\infty$.
\end{Lem}

\begin{pf}
As $\Gamma_{j}$ is nondecreasing and unbounded for $j = 1,\ldots,k$, it
is sufficient to show that, for all $r\in(0,\infty)$,
$\Gamma_{k}\circ\cdots \circ\Gamma_{1}(r)<\infty$ almost surely. Let
\begin{equation}
\label{theta} \Theta_{k}(r):=\Gamma_{k}\circ\cdots\circ
\Gamma_{1}(r).
\end{equation}

We will show by induction that
\begin{equation}
\label{etheta} E\bigl(\Theta_{k}(r)\bigr) = r\sum
_{x_1 =1}^{\infty}\cdots\sum_{x_k=1}^{\infty}
\gamma_{1}(x|_1)\cdots\gamma_{k}(x|_k)
=r\sum_{x\in\mathbb{N}^k_{*}}\bar{\gamma}_{k}(x).
\end{equation}
%
Since the right-hand side of~(\ref{etheta}) is finite by assumption
[see~(\ref{eqfin}) above], this closes the argument.

Equation~(\ref{etheta}) is immediate from the definition for $k=1$. Let
us suppose that it holds up to $k-1$, for a fixed arbitrary $k\geq2$.
Let us consider the constancy intervals of $X_{k,1}$ (i.e., maximal
intervals over which $X_{k,1}$ is constant): $\mathcal{I} =
\{$constancy intervals of $X_{k,1} \subset[0,\Theta_{k}(r)]\}$. We can
enumerate such intervals as $\mathcal{I} = \{
I(s):=[\Gamma_1(s-),\Gamma_1(s)); s \in\mathcal{S}_1 \cap[0,r] \}$. So,
\begin{equation}
\label{Theta} \Theta_{k}(r) = \sum_{s \in\mathcal{S}_1 \cap[0,r]}
\bigl|I(s)\bigr| = \sum_{x_1 = 1}^{\infty}\sum
_{i_1 = 1}^{N_{r}^{(x_1,1)}}L_{i_1}^{(x_1)},
\end{equation}
where\vspace*{-2pt} $L_{i_1}^{(x_1)}:= |I(\sigma_{i_1}^{(x_1,1)})|$,
and we recall that $N^{(x_1,1)}$ is a Poisson process of rate 1.
Now\vspace*{-1pt} notice that, for every $x_1 \in\mathbb{N}_{*}$,
$N^{(x_1,1)}$ and $L^{(x_1)}:= \{ L_{i_1}^{(x_1)}\dvtx  i_1 \geq1 \}$
are independent, and $L^{(x_1)}$ is an i.i.d. family of random
variables with $L_{i_1}^{(x_1)}
\sim\Theta_{k-1}^{(x_1)}(\gamma_1(x_1)T_{1}^{(x_1)})$, where
$\Theta_{k-1}^{(x_1)} = \Gamma_{k-1} \circ\cdots\circ\Gamma_{1}$ is the
corresponding of $\Theta_k$ for a $K$ process on ${\mathbb
T}_{k}^{(x_1)}$, the $k-1$-level subtree of ${\mathbb T}_{k}$ rooted on
$x_1$, and parameter set
$\underline{\gamma_{k}^{(x_1)}}=\{\gamma_i(x_1,\cdot)$:
$i=2,\ldots,k\}$.

Then
\begin{eqnarray}
\label{et1} E\bigl(\Theta_{k}(r)\bigr) &=& r \sum
_{x_1 =1}^{\infty}E \bigl\{ \Theta_{k-1}^{(x|_1)}
\bigl(\gamma_1(x|_1) T_1^{(x_1)}
\bigr)\bigr\}
\nonumber\\[-8pt]\\[-8pt]
&=& r \sum_{x_1 = 1}^{\infty}
\gamma_{1}(x|_1) \sum_{x_2,\ldots,x_k}
\gamma_{2}(x|_2)\cdots\gamma_{k}(x|_k),\nonumber
\end{eqnarray}
where we have used that $T_{i_1}^{(x_1)}, i_1 \geq1$, are i.i.d. with
mean 1 random variables, independent of all other random variables,
and, in the second equality, the induction hypothesis. The coincidence
of the right-hand sides of~(\ref{etheta}) and~(\ref{et1}) closes the
argument for the first assertion.

It follows readily from~(\ref{Theta}) and the independence of the
summands on its right-hand side, and the fact that their distribution
depend only on $x_1$, that a.s. $\Theta_{k}(r)\to\infty$ as
$r\to\infty$, and the second assertion follows from this and the first
assertion.
\end{pf}


%
\begin{rmk}\label{rmkTheta}
We will on several occasions below, as we did right above, work
with\vspace*{-1pt} the compounded clock $\Theta_{j}$ rather than with
the simple clock $\Gamma_{j}$ or simple time. In finite volume, we will
do the same with the finite volume version $\Theta^{(n)}_{j}$
[appearing below; see~(\ref{then})]. This is (only) for convenience,
since we can obtain simpler expressions to work with for quantities
involving the compounded clocks, like~(\ref{Theta}) above,
or~(\ref{extranumber})--(\ref{extracont}) below, than ones for simple
clocks or simple time. A typical argument (as the one above) will use
the fact that a.s. $\Theta_{j}(r)<\infty$ for all $r$ and
$\Theta_{j}(r)\to\infty$ as $r\to\infty$ to go from a statement
involving $\Theta_{j}(r)$ to one involving $\Gamma_{j}(r)$ or $r$.
Notice that the definitions of $X^F_k$ and $X_k$ involve only simple
clocks $\Gamma^F_{j}$ and $\Gamma_{j}$, respectively; the composition
behind $\Theta_{j}$ (and $\Theta^{(n)}_{j}$) helps with computations,
however.
\end{rmk}

We next prove a property about the infinities of $X_k$. Even though
this result is not used in what follows it, and is in some sense
contained in the next result, Lemma~\ref{hypercube}, it sheds light on
a characteristic of $X_k$ which is worth pointing out.

%
\begin{Lem}\label{nullmeasure2}
Let $k\geq1$.
\begin{longlist}[(2)]
%
\item[(1)] The set of infinities of $X_{k}$ has a.s. Lebesgue
    measure zero. More precisely, let ${\mathfrak I}_k=\bigcup_{i =
    1}^k{\mathfrak I}_{k,i}$, with ${\mathfrak I}_{k,i}:=\{ t
    \geq0\dvtx  X_{k,i}(t) = \infty\}$; then, ${\mathfrak I}_k$ has
    a.s. Lebesgue measure zero.

\item[(2)] Almost surely,~if $X_{k,i}(t) = \infty$ for some $t \geq
0$ and $i = 1,\ldots,k$,
then $X_{k,j}(t) = \infty$ for $i \leq j \leq k$.
\end{longlist}
\end{Lem}

\begin{pf} 
From the construction of $X_k$ it follows that $X_{k,k}(t)\in{\mathbb
N}_\ast$, that is, is finite if and only
$\varphi_k(t)\in\mathcal{S}_{k}$ and $t\in \mathrm{ULC}_{\varphi_{k}}$,
which means that
$t\in\bigcup_{s\in\mathcal{S}_{k}}[\Gamma_k(s)-,\Gamma_k(s))
=\mathbb{R}^{+}\setminus\Gamma_{k}(\mathbb{R}^{+})$. From
Remark~\ref{atomb}, it follows that ${\mathfrak I}_{k,k}$ has a.s.
Lebesgue measure zero, and in particular both claims are established
for $k=1$.

Let us inductively suppose they hold for $k=K-1$ for $K\geq2$. By the
reasoning of previous paragraph, we have that ${\mathfrak I}_{K,K}$ has
a.s. Lebesgue measure zero. It is thus enough to consider ${\mathfrak
I}_{K,i}\setminus{\mathfrak I}_{K,K}$ for $i<K$. Again by the reasoning
of the previous paragraph and the construction of $X_K$, we have that
the latter set is nonempty only if ${\mathfrak
I}_{K-1,i}\cap\mathcal{S}_{K}\ne\varnothing$, but given the induction
hypothesis, this a.s. does not happen for a.e. realization of
$\mathcal{S}_{K}$, since Poisson processes of constant rate a.s. assign
no point to sets of null Lebesgue measure. This means that a.s.
${\mathfrak I}_{K}={\mathfrak I}_{K,K}$, and the induction step for the
first claim follows. The latter equality implies in particular the
second claim for $i=K-1$ (which is the only remaining case if $K=2$).
Suppose now that $K\geq3$ and $X_{K,i}(t) = \infty$ for some $i<K-1$
and $t\geq0$; since $X_{K,i}(t)=X_{K-1,i}(\varphi_K(t))$, we may apply
the induction hypothesis to conclude that
$X_{K,i}(t)=X_{K-1,i}(\varphi_K(t)) =X_{K-1,K-1}(\varphi_K(t)) =
X_{K,K-1}(t) = \infty$, and from the conclusion of the previous
sentence follows the induction step for the second claim.
\end{pf}

The next result roughly states that once a coordinate of a $K$-process
is large, then so are the subsequent ones. This is in line with the
property stated in Lemma~\ref{nullmeasure2}(2) above---in a way, it is
a continuous extension of it. In the next section we establish a finite
volume analogue; see Lemma~\ref{hypercubeprob} below.

%
\begin{Lem} \label{hypercube}
Let $X_{k}$ be a $k$-level $K$-process. Given $T > 0$, not necessarily
deterministic, and $m\geq1$, there a.s. exists $\tilde m=\tilde
m^{(k)}$ such that if $X_{k,i}(t) > \tilde m$ for some $1\leq i<k$ and
$t \in[0,T]$, then $X_{k,j}(t) > m$ for all $j = i+1,\ldots,k$.
\end{Lem}

\begin{pf}
We will start by claiming that there a.s. exists $\hat m^{(k)}$ such
that if $X_{k,i}(t) > \hat m^{(k)}$ for any $t\leq T$ and
$i=1,\ldots,k-1$, then $X_{k,k}(t) > m$. This closes the argument when
$k=2$. For $k\geq3$, we use an inductive argument.

For $m\in\mathbb{N}_{*}$ and $j=1,\ldots,k$, let
\begin{equation}
\label{tsj} \tilde{\mathcal{S}}_{j}^{(m)}=\bigl\{
\sigma_{i}^{(x,j)}\dvtx  x=1,\ldots,m, i \geq1 \bigr\}
\end{equation}
be the set of Poissonian marks of level $j$ with labels at most $m$.
Let us fix \mbox{$T' > 0$} deterministic, and let $\mathcal{T}_{kj}(l)$
denote the set of times up to $\Theta_{k-1}(T')$ spent by $X_{k-1,j}$
above $l$, $j=1,\ldots,k-1$. By a similar reasoning as the one employed
to prove~(\ref{etheta}), we may check that the expected Lebesgue
measure of $\mathcal{T}_{kj}(l)$ equals
\begin{equation}
\label{tt1} T'\sum_{x_{1}=1}^{\infty}
\cdots\sum_{x_j =l+1}^{\infty}\cdots\sum
_{x_{k-1}=1}^{\infty} \gamma_{1}(x|_1)
\cdots\gamma_{k-1}(x|_{k-1}).
\end{equation}
[We recall that the reason to work with the compounded clocks $\Theta
_j$'s rather than the simple clocks $\Gamma_j$'s or deterministic times
is precisely to be able to derive a~simple formula like the one
in~(\ref{tt1}), which would be more complicated for simple clocks or
deterministic times replacing $\Theta_{k-1}(T')$.] Since that Lebesgue
measure is decreasing in $l$, and, as follows from our assumptions on
$\underline{\gamma_k}$, the expression in~(\ref{tt1}) vanishes as
$l\to\infty$, we have that the limit of that Lebesgue measure as
$l\to\infty$ vanishes almost surely. We then have from elementary
properties of Poisson processes that
\begin{equation}
\label{tt2} \Biggl\{\bigcup_{j=1}^{k-1}
\mathcal{T}_{kj}(l)\Biggr\}\cap\tilde{\mathcal{S}}_{k}^{(m)}=
\varnothing
\end{equation}
for all large enough $l$ almost surely, so given $m\in\mathbb{N}_{*}$,
we find $\bar m^{(k)}=\bar m^{(k)}(T')$ such that on $\{\Theta_{k}(T')
> T\}$ if $X_{k,i}(t) > \bar m^{(k)}$ for some $t\leq T$ and
$i=1,\ldots,k-1$, then, since on $\{\Theta_{k}(T') > T\}$ the
trajectory of $X_{k,k}$ in $[0,T]$ depends only on the Poisson points
of $\mathcal{N}_{k}$ on $[0,\Theta_{k-1}(T')]$, we have
from~(\ref{tt2}) that $X_{k,k}(t) > m$. Since from second assertion of
Lemma~\ref{fin} $\bigcup_{T'>0}\{\Theta_{k}(T') > T\}$ has full
measure, we may a.s. choose $T'$ such $\{\Theta_{k}(T') > T\}$ occurs,
and then choose $\hat m^{(k)}=\bar m^{(k)}(T')$, and the claim at the
beginning of the proof follows.

As we have already argued, this in particular establishes the lemma for
$k=2$, by the choice $\tilde m^{(2)}=\hat m^{(2)}$. Let us assume that
the lemma is established for \mbox{$k-1\geq2$}. This means that given
$m\geq1$ there a.s. exists $\tilde m^{(k-1)}$ such that if
$X_{k-1,i}(t) > \tilde m^{(k-1)}$ for some $1\leq i<k-1$ and $t \in
[0,\Theta_{k-1}(T')]$,
then $X_{k-1,j}(t) > m$, 
where $T'$ is as at the end of the previous paragraph [notice that we
used $\Theta_{k-1}(T')$ as $T$ here].
The claim of the lemma then follows by the choice $\tilde
m^{(k)}=\tilde m^{(k-1)}\vee\hat m^{(k)}$, where $m^{(k)}$ is as in the
claim at the beginning of the proof with $\Theta_{k-1}(T')$ replacing
$T$. Indeed, if for some $t\in[0,T]$ we have $X_{k,i}(t)>\tilde
m^{(k)}$, then, by the claim at the beginning, $X_{k,k}(t)>m$, and the
claim of the lemma is established for $j=k$. If $i<j<k$, then since the
trajectory of $(X_{k,i}$, $i=1,\ldots,k-1)$ in $[0,T]$ shadows that of
$X_{k-1}$ in $[0,T'']$ for some $T''\leq\Theta_{k-1}(T')$, meaning that
there exists $t'\in[0,\Theta _{k-1}(T')]$ such that $(X_{k,i}(t),
i=1,\ldots,k-1)=X_{k-1}(t')$, we have that
$X_{k,j}(t)=X_{k-1,j}(t')>m$, by the induction hypothesis, and the
argument is complete.
\end{pf}


\section{Convergence}
\label{convergence}

\subsection{Scaling limit for the GREM-like trap model}
\label{scalim}

We start this section with our main result, the scaling limit for the
GREM-like trap model in the fine tuning regime and extreme time scale.
Let us go again, this time in more detail,
over the definition of these terms; see the last paragraph of
Section~\ref{tree}. [In this section, we replace the notation above
with superscript ``$F$,'' denoting finite volume, to a~notation with
superscript ``$(n)$,'' to emphasize sequence dependence instead.] The
parameters of the model of this subsection will be taken random, as
described below.

For $j = 1,\ldots,k$, let $\tau_{j}:= \{ \tau_{j}(x|_j); x|_j \in
\mathcal{M}|_{j}\}$ be an i.i.d. family of random variables in the
domain of attraction of an $\alpha_{j}$-stable law. We suppose
\begin{equation}
0 < \alpha_{1} < \cdots< \alpha_{k} < 1.
\end{equation}


For $j = 1,\ldots,k$, we will relabel $\tau_{j}$ obtaining
$\tau_{j}^{(n)} = \{\tau_{j}^{(n)}(x|_{j}); x|_{j}\in\mathcal
{{M}}|_{j}\}$,
so that, for every $(x|_{j-1}) \in\mathcal{{M}}|_{j-1}$,
$\{ \tau_{j}^{(n)}(x|_{j}); x_{j} \in\mathcal{{M}}_{j} \}$
are the decreasing order statistics of
$\{\tau_{j}(x|_j); x_j \in\mathcal{M}_{j}\}$.

For $j = 1,\ldots,k$, $n \geq1$ and $x|_{j-1} \in\mathbb
{N}_{*}^{j-1}$, let
\begin{equation}
c_{j}^{(n)} = \bigl(G_j^{-1}
\bigl(M_{j}^{-1}\bigr)\bigr)^{-1},
\end{equation}
where\vspace*{1pt} $G_j^{-1}$ is the (generalized) inverse of $G_j\dvtx
[0,\infty)\to (0,1]$ such that $G_j(t)=P(\tau_{j}(x|_{j}) > t)$,
and make
\begin{equation}
\label{eqgjn} \gamma_{j}^{(n)}(x|_{j})=
c_{j}^{(n)} \tau_{j}^{(n)}(x|_j),\qquad
x|_j\in\mathcal{M}_{j};
\end{equation}
%
let also
\begin{equation}
\label{eqgj} \bigl\{\gamma_{j}(x|_{j}), x_j
\in\mathbb{N}_{*}^j\bigr\}
\end{equation}
denote independent $j$-parametrized Poisson point processes,
with intensity measure given by $y^{-\alpha_{j}-1}$, $y>0$, in
decreasing order.


The \textit{fine tuning} regime mentioned above and at the
\hyperref[intro]{Introduction} corresponds to choosing $M^{(n)}_{1} =
n$ and
\begin{equation}
\label{eqft} M_{j+1}^{(n)} = \bigl\lfloor
1/c_{j}^{(n)} \bigr\rfloor,\qquad j = 1,\ldots,k-1.
\end{equation}
We will make this choice from now on.

Let
\begin{equation}
\label{eqtg} \tilde\gamma_{j}^{(n)}(x|_{j})=
\cases{ \gamma_{j}^{(n)}(x|_{j}), &\quad if $j=1,\ldots,k-1$,
\vspace*{3pt}\cr
\tau_{k}^{(n)}(x|_k), &\quad if
$j=k$}
\end{equation}
and let $\tilde X_{k}^{(n)} \sim TM({\mathbb T}^{(n)}_k; \underline
{\tilde\gamma^{(n)}_{k}})$.

%
\begin{rmk}
\label{eqtxn}
One may readily check from Lemma~\ref{lmrep} that, in
terms of the coin tossing description, $\tilde X_{k}^{(n)} \sim
TM({\mathbb T}^{(n)}_k; \tau_{k}^{(n)},(p^{(n)}_{j})_{j=1}^{k-1})$,
where for $j=1,\ldots,k-1$ and $x|_{j}\in\mathcal{{M}}|_j$
\begin{equation}
\label{pn} p^{(n)}_{j}(x|_{j}) =
\frac{1}{1 + \tau^{(n)}_{j}(x|_{j})}.
\end{equation}
With this description, and general finite $M_1,\ldots,M_k$ [not
necessarily satisfying~(\ref{eqft})], we call $\tilde X_{k}^{(n)}$ the
GREM-like trap model on ${\mathbb T}^{(n)}_k$ with parameters
$\tau_{j}(x|_{j})$, $j=1,\ldots,k$, $x|_{j}\in\mathcal{{M}}|_j$. [The
relabeling performed in this subsection (cf. the definition given in
the last paragraph of Section~\ref{tree}) is necessary for the
existence of the limit.] In this guise, with a choice of
$M_1=\cdots=M_k$, the model was introduced and studied
in~\cite{knsn1,knsn2}, with the derivation of infinite volume aging
functions as the main motivation, with infinite volume limits taken
first, and then an infinite time limit. See Remark~\ref{rmksn} below.
\end{rmk}

Let us speed up $\tilde X_{k}^{(n)}$ by $c_{k}^{(n)}$, namely, let
\begin{equation}
{X}_{k}^{(n)} =\tilde X_{k}^{(n)}
\bigl(t/c_{k}^{(n)}\bigr),\qquad t \geq0.
\end{equation}
This corresponds to the extreme time scale mentioned above and at the
\hyperref[intro]{Introduction}. One may readily check that $X_{k}^{(n)}
\sim TM({\mathbb T}^{(n)}_k; \underline{\gamma^{(n)}_{k}})$. Let $X_{k}
\sim K({\mathbb T}_k;\underline{\gamma_{k}})$.

%
\begin{theo}\label{scal}
Let $X_{k}^{(n)}$ and $X_{k}$ be as above. Then
\begin{equation}
\bigl({X}_{k}^{(n)}, \underline{\gamma_{k}^{(n)}}
\bigr)\quad \Rightarrow\quad (X_{k}, \underline{\gamma_{k}}),
\end{equation}
where $\Rightarrow$ means weak convergence in the product of Skorohod
space with
the space of finite measures in $\mathbb{N}_{*}^{k}$ equipped with the topology
of weak convergence.
\end{theo}

The Skorohod space in the above statement will be described in detail
at the beginning of next subsection.

%
\begin{rmk}
\label{rmksn} As a note on the differences between the above result and
those of~\cite{knsn1,knsn2}, let us point out that the choice of volume
relations should not be very important in the context
of~\cite{knsn1,knsn2}, since the volume limit is taken first, and then
the time limit. One expects aging to take place in this regime, and
that is what is behind the (explicit) results of~\cite{knsn1,knsn2}.
Our choice of volume/time relations is on the other hand essential in
order to obtain the specific limit stated above. In particular, they
represent not an aging time regime, but an ergodic time regime, that
is, a time regime where the process is already close to equilibrium.
(Aging is a~phenomenon that instead takes place far from equilibrium.)
In this sense, our results do not compare immediately to those
in~\cite{knsn1,knsn2}, since they involve different time/volume
regimes, where different behaviors take place. In~\cite{kngg}, a
smaller time regime is studied, where aging takes place, with results
comparable to~\cite{knsn1,knsn2}. Other choices of volume/time scaling
may lead to different asymptotics (from the above one and conceivably
also from~\cite{knsn1,knsn2}).
\end{rmk}

\subsection{Infinite volume limit for the $k$-level trap model}
\label{infvol}

As anticipated in the \hyperref[intro]{Introduction},
Theorem~\ref{scal} will be proven in Section~\ref{GREM} below by
verifying the conditions of an infinite volume limit result for
$k$-level trap models. This is the object of this and the next two
subsections. We may in this section, and in the following two
subsections, think of the parameters of the model as deterministic. We
will return to random parameters at the last subsection.

Let us consider a sequence of $k$-level trap models $X_{k}^{(n)}, n
\geq1$, on a sequence of finite\vspace*{-1pt} trees ${\mathbb
T}_{k}^{(n)}$, with volumes $M_1=M_1^{(n)},\ldots,M_k=M_k^{(n)}$,
and
parameter sets~$\underline{\gamma_{k}^{(n)}}$,\vadjust{\goodbreak} respectively
(see\vspace*{-1.5pt} Definition~\ref{klt}), and prove a weak
convergence result for that sequence under the Skorohod topology on
$D({\bar{\mathbb N}}_{*}^{k}, [0,\infty))$, the space of c\`adl\`ag
functions from $[0,\infty)$ to ${\bar{\mathbb N}}_{*}^{k}$. As
anticipated at the beginning of Section~\ref{scalim}, we replace the
superscript ``$F$'' used in the first sections by ``$(n)$'' everywhere
to emphasize the dependence on $n$.

Before proceeding, let us briefly review the Skorohod topology.
We start by equipping ${\bar{\mathbb N}}_{*}^{k}$ with the metric
\begin{equation}
d(x,y) = \max_{1\leq j\leq k}\bigl|x_{j}^{-1} -
y_{j}^{-1}\bigr|,\qquad x, y \in\bar{\mathbb{N}}_{*}^{k},
\end{equation}
where $\infty^{-1}=0$, under which it
is compact.
The Skorohod metric on $D({\bar{\mathbb N}}_{*}^{k},\break  [0,\infty))$
is as follows. For $f,g\in D({\bar{\mathbb N}}_{*}^{k}, [0,\infty))$,
let
\begin{equation}
\label{eqmet1} \rho(f,g) = \inf_{\lambda\in\Lambda} \biggl[\phi(\lambda)
\vee\int_0^\infty e^{-u} \rho(f,g,\lambda,u) \,du \biggr],
\end{equation}
where
\begin{equation}
\label{eqmet2} \rho(f,g,\lambda,u) = \sup_{t\geq0}\,d\bigl(f(t
\wedge u),g\bigl(\lambda(t)\wedge u\bigr)\bigr)
\end{equation}
with $\Lambda$ the class of \textit{time distortions}: increasing
Lipschitz functions from $[0,\infty)$ onto $[0,\infty)$, and $\phi\dvtx
\Lambda\to[0,\infty)$ such that
\begin{equation}
\label{eqg} \phi(\lambda)=\sup_{0\leq s<t}\biggl\llvert\log
\frac{\lambda_t-\lambda
_s}{t-s}\biggr\rrvert;
\end{equation}
see Section~3.5 in~\cite{knek}.

In order to get our convergence result, we will impose the following
conditions on the
volumes and parameters. For $j=1,\ldots,k$, suppose that as $n
\rightarrow\infty$
\begin{eqnarray}
\label{em1} M_{j}^{(n)}&\to&\infty,
\\
%
\label{em2} \gamma_j^{(n)}(x) &\rightarrow&
\gamma_j (x)\qquad\mbox{for every } x\in{\mathbb N}_{*}^{j}\mbox{ and }
\sum_{x\in{\mathbb N}_{*}^{j}}\bar\gamma_j^{(n)}(x)
\to\sum_{x\in
{\mathbb N}_{*}^{j}}\bar\gamma_j(x)\hspace*{-35pt} 
\end{eqnarray}
with $\gamma_{j}, \bar\gamma_{j}$ as in the beginning of
Section~\ref{GKprocess} [see paragraph of~(\ref{gamma}), (\ref{eqfin})
above], $\gamma_j^{(n)}\equiv0$ on ${\mathbb
N}_{*}^{j}\setminus\mathcal {M}|_{j}$ and
\begin{equation}
{\bar\gamma}^{(n)}_{j}(x|_{j}):=
\gamma^{(n)}_{1}(x|_1)\times\gamma
^{(n)}_{2}(x|_2) \times\cdots\times
\gamma^{(n)}_{j}(x|_{j}).
\end{equation}


Our result will require additional conditions that look quite
intricate. We state
them now and discuss them, together with the above conditions, after we
state the convergence result.
We further suppose that for $j=2,\ldots,k$
\begin{eqnarray}\label{em4}
&& \frac{1}{\prod_{p=1}^{j-1}M_{p+1}}
\sum_{l=1}^{j-1}\ \sum_{x|_{j-1}\in\mathcal{M}|_{j-1}}
\ \prod_{p=1}^{l-1}
\bigl(M_{p+1} \gamma_p^{(n)}(x|_{p})
\bigr)
\nonumber\\[-9pt]\\[-9pt]
&&\hspace*{129pt}{}\times  \prod_{p = l + 1}^{j-1}\bigl( 1 +
M_{p+1} \gamma_{p}^{(n)}(x|_{p})\bigr)
\rightarrow0\nonumber
\end{eqnarray}
and
\begin{eqnarray}
\label{em3} 
&& \frac{1}{\prod_{p=1}^{j-1}M_{p+1}} 
\sum_{l=1}^{j-1}
\ \sum_{x|_j\in\mathcal{M}|_{j}}\gamma_j^{(n)}(x|_j)
\ \prod_{p=1}^{l-1} \bigl(M_{p+1}
\gamma_{p}^{(n)}(x|_{p})\bigr)
\nonumber\\[-9pt]\\[-9pt]
&&\hspace*{109pt}{}\times
\prod_{p = l + 1}^{j-1}\bigl(1 + M_{p+1}
\gamma_{p}^{(n)}(x|_{p})\bigr) \rightarrow0\nonumber
\end{eqnarray}
as $n \rightarrow\infty$, where by convention
\[
\prod_{p=1}^{0}
\bigl(M_{p+1} \gamma_{p}^{(n)}(x|_{p})
\bigr) = \prod_{p = j }^{j-1}\bigl( 1 +
M_{p+1} \gamma_{p}^{(n)}(x|_{p})\bigr)
= 1.
\]
Here, and many times below, we omit the superscript ``$(n)$'' from the
notation for the
volumes $M_1,\ldots,M_k$.

We are ready to state our infinite volume limit result.

%
\begin{theo}\label{GK}
For\vspace*{-1pt} $n\geq1$, let $X_{k}^{(n)}$ be the trap model on
${\mathbb T}_{k}^{(n)}$, with volumes $M_1^{(n)},\ldots, M_k^{(n)}$,
and parameter sets $\underline{\gamma_{k}^{(n)}}$, respectively,
satisfying conditions (\ref{em1})--(\ref{em3}). Let\vspace*{-1.5pt}
$X_{k}$ be the $K$-process on ${\mathbb T}_{k}$ with parameter set
$\underline{\gamma_{k}}$. Then $X_{k}^{(n)}\rightarrow X_{k}$ weakly in
Skorohod space as $n \rightarrow\infty$.
\end{theo}

We will see (from the proofs) that conditions (\ref{em1})--(\ref{em3})
have the following significance. Obviously, (\ref{em1}) means that we
are taking an infinite volume limit. Equation (\ref{em2}) implies that
the contributions coming from the Poisson marks to the construction of
$X_{k}^{(n)}$ converge (in a uniform way) to the respective
contributions of
(Poisson) marks 
of $X_{k}$. Finally, as will be seen in the arguments below,
(\ref{em4})--(\ref{em3}) imply the negligibility of the total
contribution of the extra marks entering $X_{k}^{(n)}$. [Poisson and
extra marks were introduced in the paragraph before~(\ref{xi}) above.]
It may be readily checked that in general neither are conditions
(\ref{em1})--(\ref{em3}) equivalent, nor do they follow from previous
conditions; in the generality of the statement of Theorem~\ref{GK},
indeed, they need to be separately imposed.

%
\begin{rmk}
\label{cond} One way to gain insight into the meaning
of~(\ref{em4})--(\ref{em3}) is as follows. In order to have a single
condition, we start by writing the sum over $\mathcal{M}|_{j-1}$
in~(\ref{em4}) as sum over $\mathcal{M}|_{j}$ with an extra term of
$1/M_j$ multiplying each summand. We then sum the resulting
expression\vadjust{\goodbreak}
to the one on the left of~(\ref{em3}), getting
\begin{eqnarray}
\label{em6}
&& \frac{1}{\prod_{p=1}^{j-1}M_{p+1}} 
\ \sum_{l=1}^{j-1}
\ \sum_{x|_j\in\mathcal{M}|_{j}}
\ \prod_{p=1}^{l-1}
\bigl(M_{p+1} \gamma_{p}^{(n)}(x|_{p})
\bigr)\nonumber
\\
&&\hspace*{114pt}{}\times
\prod_{p = l + 1}^{j-1}\bigl(1 +
M_{p+1} \gamma_{p}^{(n)}(x|_{p})\bigr)
\\
&&\hspace*{149pt}{}\times\biggl(\frac{1}{M_{j}}+\gamma_{j}^{(n)}(x|_{j})
\biggr).\nonumber
\end{eqnarray}

Dividing now the double product inside the double sum in~(\ref{em6}) by
the product outside the double sum, and defining
\begin{eqnarray}
\bar\gamma_{j,l}^{(n)}(x|_{j})&:=&\prod
_{p=1}^{l-1}\gamma_p^{(n)}(x|_{p})
\frac{1}{M_{l+1}}
\nonumber\\[-10pt]\\[-10pt]
&&{}\times \prod_{p = l + 1}^{j-1} \biggl(
\frac{1}{M_{p+1}}+\gamma_{p}^{(n)}(x|_{p})
\biggr) \biggl(\frac{1}{M_{j}}+\gamma_{j}^{(n)}(x|_{j})
\biggr),\nonumber
\end{eqnarray}
we find that~(\ref{em4})--(\ref{em3}) are equivalent to the following
condition. For $1\leq l<j\leq k$, as $n\to\infty$
\begin{equation}
\label{em5} \sum_{x|_{j}\in\mathcal{M}|_{j}}\bar\gamma_{j,l}^{(n)}(x|_{j})
\to0.
\end{equation}

Compare $\bar\gamma_{j,l}^{(n)}(x|_{j})$ to $\bar\gamma
_j^{(n)}(x|_{j})$ and~(\ref{em5}) to the second condition
in~(\ref{em2}).
\end{rmk}

The remainder of this section is organized as follows. We briefly start
below, in this same subsection, with the proof of Theorem~\ref{GK}. The
full proof will require a number of auxiliary results, which we collect
in Section~\ref{prel} below, before proceeding with the proof in
Section~\ref{pf} after that. And, as we already mentioned,
Section~\ref{GREM} is devoted to the proof of Theorem~\ref{scal}.

\begin{pf*}{Proof of Theorem~\ref{GK}}
We will argue by induction, using coupled versions of $X_{k}^{(n)}$ and
$X_{k}$, and show convergence in probability for a subsequence. The
coupling is going to be given by using common Poisson processes $\{
N^{(x_i,i)}, x_i \in\mathbb{N}_{*}, i = 1,\ldots,k\}$ and common
exponential variables $\{ T_{s}^{(i)}, s \in\mathbb{R}^{+}, i =
1,\ldots,k\}$ in the construction of $X_{k}^{(n)}$ and $X_{k}$. The
notation is detailed at the beginning of Section~\ref{pf}.
It will be clear that the same can be done for every subsequence of
$(n)$, and that the limiting distribution for each subsubsequence does
not depend on the subsequence. This then implies weak convergence of
the original sequence.

Lemma 3.11 of~\cite{knfm} establishes the (convergence in probability;
actually a.s. convergence) result for $k=1$ and $\gamma_{1}^{(n)}(x)$
not depending on $n$ as soon as $x\leq M_1$. This result (convergence
in probability) holds (with minor changes in argumentation, as sketched
in the proof of Theorem 5.2 of~\cite{knfm}) in our case as well. It is
also part of the argumentation of Lemma 3.11 and Theorem 5.2
of~\cite{knfm}, and can also be readily checked independently, that for
every $r\in[0,\infty)$,
\begin{equation}
\label{eqgn} \Gamma_{1}^{(n)}(r) \rightarrow
\Gamma_{1}(r)
\end{equation}
in probability as $n\to\infty$.

As part of our induction argument, we will then assume that for
$j=1,\ldots,k-1$ and every $r\in[0,\infty)$,
\begin{eqnarray}
\label{convxj} X_{j}^{(n)}&\rightarrow& X_{j},
\\
\label{convgj} \Gamma_{j}^{(n)}(r)& \rightarrow&
\Gamma_{j}(r)
\end{eqnarray}
as $n \rightarrow\infty$ almost surely, possibly over a subsequence.
\end{pf*}


\subsection{\texorpdfstring{Auxiliary results for the proof of Theorem~\protect\ref{GK}}
{Auxiliary results for the proof of Theorem 5.4}}\label{prel}

We assume throughout that the hypotheses of Theorem~\ref{GK} are in
force.

Our first auxiliary result establishes that the contribution of extra
marks and their descendants to $X_{j}^{(n)}$ is negligible as $n
\rightarrow\infty$. That is the content of Lemma~\ref{extramark}. To be
precise, let $\mathcal{E}_{2}^{(n)} = \mathcal{R}_{2}^{(n)}$ and for
$3\leq i \leq k$
\begin{equation}
\mathcal{E}_{i}^{(n)} = \mathcal{R}_{i}^{(n)}
\cup\bigl\{ s \in\mathcal{S}_{i}^{(n)}\dvtx
\varphi_{i - 1}^{(n)}(s) \in\mathcal{E}_{i-1}^{(n)}
\bigr\}.
\end{equation}
$\mathcal{E}_{i}^{(n)}$ represents the extra marks of the $i$th level
and the \textit{descendants} of extra marks of previous levels in the
$i$th level (i.e., Poisson marks belonging to constancy intervals
originating from extra marks of the previous level or descendants of
extra marks from levels before that).

%
\begin{Lem} \label{extramark}
Assume that the induction hypotheses~(\ref{convxj})--(\ref{convgj})
hold. Then, for every $r > 0$ and $j=2,\ldots,k$, we have that:
\begin{longlist}[(a)]
\item[(a)] $\min\{ \xi_j^{(n)}(s)\dvtx  s \in\mathcal{E}_{j}^{(n)}
    \cap[0,r] \} \rightarrow\infty$ in probability as $n
    \rightarrow\infty$.

\item[(b)] $\mu^{(n)}_j(\mathcal{E}^{(n)}_{j} \cap[0,r])
\rightarrow0$ in
probability as $n \rightarrow\infty$.
\end{longlist}
\end{Lem}

\begin{pf}
For $r > 0$, $j=1,\ldots,k$, let
\begin{equation}
\label{then} \Theta_{j}^{(n)}(r):=\Gamma_{j}^{(n)}
\circ\cdots\circ\Gamma_{1}^{(n)}(r)
\end{equation}
and define\vspace*{-2pt}
$K_{j}^{(n)}(r):=|\mathcal{E}_{j}^{(n)}\cap[0,\Theta_{j-1}^{(n)}(r)]|$,
where (here) $|\cdot|$ stands for cardinality. An evaluation of
$K_{j}^{(n)}(r)$ will play a crucial role in the proof. We begin with
that.

It follows from induction hypothesis~(\ref{convgj}) that for
$j=1,\ldots,k-1$,\break $\Gamma_{j}^{(n)}(r)\to\infty$ as $r\to\infty$
in probability, uniformly in $n$. So it is enough to consider
$\mathcal{E}^{(n)}_{j} \cap[0,\Theta_{j-1}^{(n)}(r)]$ instead of
$\mathcal{E}^{(n)}_{j} \cap[0,r]$. In\vspace*{-1pt} order to evaluate
the cardinality of that set, as well as its contribution to
$\mu_{j}^{(n)}([0,\Theta_{j-1}^{(n)}(r)])$, we start by describing the
structure of $\mathcal{H}_1^{(n)}: = \mathcal{S}_1^{(n)}$,
$\mathcal{H}_2^{(n)}: = \mathcal{S}_2^{(n)}
\cup\mathcal{R}_2^{(n)},\ldots, \mathcal{H}_k^{(n)}: =
\mathcal{S}_k^{(n)} \cup\mathcal{R}_k^{(n)}$; at the same time, we will
relabel the marks of those sets conveniently.

Each $s_1 \in\mathcal{S}_{1}^{(n)}$ can be put in a one-to-one
correspondence with its label $\xi^{(n)}_1(s_1)=x_1$ and \textit{index}
$i_1(s_1) = i_1\in\mathbb{N}_{*}$ via\vspace*{-1pt} the relation $s_1 =
\sigma_{i_1}^{(x_1,1)}$. Using this correspondence, we see that to each
mark $s_1$ of $\mathcal{S}_{1}^{(n)}$ there corresponds an interval
$I_{n}^{(x_1,i_1)}$ of $\mathbb{R}^{+}$ of length $L_{n}^{(x_1, i_1)} =
\gamma_1^{(n)}(x_1) T^{(x_1,i_1)}$. Such intervals form a partition of
$\mathbb{R}^{+}$, and the random variables involved are independent
when we vary $s_1$.

Now to each $s_1 \in\mathcal{S}_{1}^{(n)}$, there corresponds marks of
$\mathcal{S}_{2}^{(n)}$ belonging to the respective interval
$I_{n}^{(x_1,i_1)}$, whose cardinality is a geometric random variable
$G^{(x_1, i_1)}$ with mean $M_2 \gamma_1^{(n)}(x_1)$, plus a mark of
$\mathcal{R}_2^{(n)}$ at the left endpoint of
$I_{n}^{(x_1,i_1)}$---recall Remark~\ref{geonumber}. Each such mark
will be identified with $(x_1,i|_2)$, where $(x_1,i_1)$ is the
identifier of $s_1$, and $i_2\in\{1,\ldots,G^{(x_1,i_1)} + 1\}$, and we
attach to it a random variable $U^{(x_1,i|_2)}$ with uniform
distribution in $\mathcal{M}_2$, which corresponds to
$\xi^{(n)}_2(s_2)$, for $(x_1,i|_2)\equiv s_2
\in\mathcal{S}_2^{(n)}\cup\mathcal{R}_2^{(n)}$. We will\vspace*{1pt}
identify the unique mark of $\mathcal{R}_2^{(n)}$ at the left endpoint
of $I_n^{(x_1,i_1)}$ with $(x_1,i_1, 1)$.

We now proceed inductively. For $3\leq j \leq k$, we assume we have
identified each mark of $\mathcal{S}_{j-1}^{(n)} \cup
\mathcal{R}_{j-1}^{(n)}$ as $(x_1,i|_{j-1})$, with $x_1 \in
\mathcal{M}_{1}$, $i_1 \geq1$, $i_l = 1,\ldots,1+G^{(x_1,i|_{l-1})}$,
$l = 2,\ldots,j-1$, where, for $l\geq3$, $G^{(x_1,i|_{l-1})}$ is
geometric with mean $M_{l} \gamma_{l-1}^{(n)}(x_1, U^{(x_1,
i|_2)},\ldots,U^{(x_1, i|_{l-1})})$, with $U^{(x_1, i|_j)} \sim$
Uniform$(\mathcal{M}_{j})$. The random variables of
\[
\mathcal{U}_{j-1}:=\bigl\{ U^{(x_1, i|_l)}\dvtx  l = 2,\ldots,j-1;
x_1,i_1,\ldots,i_l \geq1 \bigr\}
\]
are independent, and, given $\mathcal{U}_{j-1}$, so are those of
\[
\bigl\{ G^{(x_1,i|_l)}\dvtx  l = 1,\ldots,j-1; x_1,i_1,
\ldots,i_l \geq1 \bigr\}.
\]
%
Notice that $G^{(x_1,i|_j)}$ is independent of $U^{(x_1, i|_l)}$ as
soon as $j<l$. Here $i_{j-1} = 1 $ means $(x_1,i|_{j-1}) \in\mathcal
{R}_{j-1}^{(n)}$; otherwise, $(x_1,i|_{j-1})
\in\mathcal{S}_{j-1}^{(n)}$. Then to each mark $(x_1,i|_{j-1})$ there
corresponds an interval $I_n^{(x_1,i|_{j-1})}$ of $\mathbb{R}^{+}$ of
length
\[
L_n^{(x_1,i|_{j-1})}:=\gamma_{j-1}^{(n)}
\bigl(x_1,U^{(x_1, i|_2)},\ldots,U^{(x_1,i|_{j-1})}\bigr)
T^{(x_1,i|_{j-1})}
\]
with $\{ T^{(x_1,i|_{j-1})}\}$ i.i.d. mean 1 exponential random
variables independent of $\{ G^{(x_1, i|_l)}$, $U^{(x_1, i|_l)} \}$,
$l=1,\ldots,j-1$, such that $\{I_n^{(x_1,i|_{j-1})}\}$ is a partition
of $\mathbb{R}^{+}$. The mark of $\mathcal{R}_j^{(n)}$ placed at the
left end of $I_n^{(x_1,i|_{j-1})}$ is labeled $(x_1,i|_{j-1}, 1)$, and
the marks of $\mathcal{S}_j^{(n)} \in I_n^{(x_1,i|_{j-1})}$, if any,
are labeled $(x_1,i|_j), i_j = 2,\ldots, 1+G^{(x_1,i|_{j-1})}$, with
$G^{(x_1,i|_{j-1})}$ a geometric random variable with mean $M_{j}
\gamma_{j-1}^{(n)}(x_1, U^{(x_1, i|_2)}, \ldots, U^{(x_1,
i|_{j-1})})$. The random variables in $\{G^{(x_1, i|_{j-1})} \}$
are\vspace*{1pt} independent among themselves, and independent of the
previous random variables. Finally,\vspace*{1pt} we assign to each
$(x_1,i|_j)$ a random variable $U^{(x_1,i|_j)}$ uniformly distributed
in $\mathcal{M}_j$, corresponding to $\xi^{(n)}_j(s_j)$, for
$(x_1,i|_j)\equiv s_j\in\mathcal{S}_j^{(n)}\cup\mathcal{R}_j^{(n)}$,
with $\{U^{(x_1,i|_j)}\}$ independent\vspace*{1pt} among themselves,
and independent of previous random variables.

With this representation, we have labeled the marks of
$\mathcal{H}_j^{(n)}\cap[0,\Theta_{j-1}^{(n)}(r)]$, $j=1,\ldots,k$, as
$(x_1,i|_{j})$, $x_1 = 1,\ldots,M_{1}$; $i_1 = 1,\ldots,N_r^{(x_1)};
i_l = 1,\ldots,\break  1+G^{(x_1,i|_{l-1})}, 2\leq l \leq j$, with
$G^{(x_1,i|_l)}$ geometric 
with mean $M_2 \gamma^{(n)}_1(x_1)$ when $l=1$, and with mean $M_{l+1}
\gamma^{(n)}_{l}(x_1,U^{(x_1,i_1)}, \ldots,U^{(x_1,i|_l)})$, when $l =
2,\ldots,j$, respectively. $U^{(x_1,i|_l)}$ is uniformly distributed on
$\mathcal{M}_l$, $l = 2,\ldots,j$. The random variables in the family
\[
\mathcal{U}_j:= \bigl\{ U^{(x_1,i|_l)}; x_1,i_1,
\ldots,i_l \geq1, l = 2,\ldots,j \bigr\}
\]
are independent, and given $\mathcal{U}_j$ so are those in
\[
\bigl\{ G^{(x_1,i|_l)}; x_1,i_1,
\ldots,i_l \geq1, l = 1,\ldots,j \bigr\}.
\]
Notice that, as before, $G^{(x_1,i|_j)}$ is independent of $U^{(x_1,
i|_l)}$ as soon as $j<l$.

The\vspace*{-2pt} marks of $\mathcal{E}_j^{(n)} \cap[0,
\Theta_{j-1}^{(n)}(r)]$ are those $(x_1,i|_{j})$ as above for which
$i_l = 1 $ for some $l = 2,\ldots,j$. In order\vspace*{-1pt} to write
an expression for $K_j^{(n)}(r)$, we first view $\mathcal{E}_j^{(n)}
\cap[0, \Theta_{j-1}^{(n)}(r)]$ as the\vspace*{-1pt} leaves of a forest
(see Figure~\ref{figforest}), the distinct trees of which have
the\vspace*{1pt} marks labeled $(x_1,i|_{l}, 1), l = 1,\ldots,j-1$,
$i_1=1,\ldots,N_r^{(x_1)}$; $i_l=2,\ldots,\tilde
G^{(x_1,i|_{l-1})}:=1+G^{(x_1,i|_{l-1})}, l=2,\ldots,j-1$, as roots;
the tree rooted at $(x_1,i|_{l}, 1)$ consisting of, besides the root,
marks whose labels form the set
\begin{equation}
\label{frak} 
\mathfrak{T}_j^{(x_1,i|_l)}:=
\bigl\{\bigl(x_1,i|^l_j\bigr)\dvtx  
i_m=1,\ldots,\tilde{G}^{(x_1,i|^l_{m-1})}, m=l+2,\ldots,j
\bigr\},
\end{equation}
where, for $1\leq h\leq j$,
\begin{equation}
\label{ilh} i|^l_h= \cases{ i|_l, &
\quad if $h\leq l$,
\vspace*{2pt}\cr
(i|_l,1), &\quad if $h=l+1$,
\vspace*{2pt}\cr
(i|_l,1,i_{l+2},\ldots,i_h), &\quad if
$h>l+1$.}
\end{equation}
%
Equation (\ref{frak}) is well defined whenever $l<j-1$; otherwise, each
of the above mentioned trees consists of its root only.

%
\begin{rmk}
\label{rmkfrak}
For each $l=1,\ldots,j-1$, the roots of the above
trees, namely the points labeled $(x_1,i|_{l}, 1)$, with $x_1$ and
$i|_{l}$ as described above, represent the extra marks of level $l+1$,
as described in the paragraph before the one containing~(\ref{xi}), now
with a labeling suited to the computations to be performed below. The
sites other than themselves on the trees of which they are the roots
represent their descendants, corresponding to either Poissonian or
extra marks originating of an extra mark at some level above.
\end{rmk}

%
\begin{figure}

\includegraphics{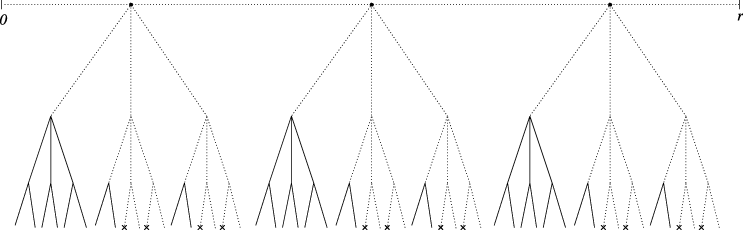}

\caption{Schematic\vspace*{-1pt} representation of portions of $\bigcup
_{i=1}^4\mathcal{H}_i^{(n)}$
and the forest whose leaves are $\mathcal{E}_4^{(n)} \cap[0, \Theta
_{3}^{(n)}(r)]$. Full points on horizontal dotted line
represent $\mathcal{H}_1^{(n)}\cap[0, r]$. Successive generations of
trees attached to each point of
$\mathcal{H}_1^{(n)}\cap[0,r]$ represent\vspace*{-1pt}
$\mathcal{H}_2^{(n)}\cap[0,\Theta_{1}^{(n)}(r)]$, $\mathcal
{H}_3^{(n)}\cap[0, \Theta_{2}^{(n)}(r)]$
and $\mathcal{H}_4^{(n)}\cap[0,\Theta_{3}^{(n)}(r)]$, respectively.
Forests of extra marks and their descendants are shown in full lines
and crosses. (Actual picture should look less regular,
since the degrees of the vertices of the trees are independent random
variables, which should be moreover large for large $n$.)}\label{figforest}
\end{figure}

Then the number of elements of $\mathcal{E}_j^{(n)} \cap[0, \Theta
_{j-1}^{(n)}(r)]$ on
the leaves of $\mathfrak{T}_j^{(x_1,i|_l)}$, $l=1,\ldots,j-2$, $j\geq
3$, is given by
\begin{equation}
\label{eqnum} \sum_{i_{l+2}=1}^{\tilde{G}^{(x_1,i|_{l+1}^l)}}\cdots\sum
_{i_{j}=1}^{\tilde{G}^{(x_1,i|_{j-1}^{l})}}1
\end{equation}
and their contribution to $\mu_j^{(n)}([0,\Theta_{j-1}^{(n)}(r)])$
amounts to
\begin{equation}
\label{eqcont} \sum_{i_{l+2}=1}^{\tilde{G}^{(x_1,i|_{l+1}^l)}}\cdots\sum
_{i_{j}=1}^{\tilde{G}^{(x_1,i|_{j-1}^l)}} \gamma^{(n)}_{j}
\bigl(x_1,U^{(x_1,i|^l_2)}, \ldots,U^{(x_1,i|^l_j)}\bigr)
T^{(x_1,i|^l_j)},
\end{equation}
where
$\{ T^{(x_1,i|_j)} \}$ are i.i.d. mean 1 exponential random variables,
independent of all other random variables.
So 
the size of
$\mathcal{E}_j^{(n)} \cap[0, \Theta_{j-1}^{(n)}(r)]$ is given by
\begin{equation}
\label{extranumber} \qquad K_j^{(n)}(r) = \sum
_{l=1}^{j-1}\sum_{x_1=1}^{M_{1}}
\sum_{i_1=1}^{N_r^{(x_1,1)}}\sum
_{i_{2}=2}^{\tilde G^{(x_1,i_1)}}\cdots\sum_{i_{l}=2}^{\tilde
G^{(x_1,i|_{l-1})}}
\sum_{i_{l+2}=1}^{\tilde{G}^{(x_1,i|_{l+1}^l)}}\cdots\sum
_{i_{j}=1}^{\tilde{G}^{(x_1,i|^l_{j-1})}}1 
\end{equation}
and its contribution to $\mu_j^{(n)}([0,\Theta_{j-1}^{(n)}(r)])$
amounts to
\begin{eqnarray}\label{extracont}
&& \mu_j^{(n)}\bigl(\mathcal{E}_{j}^{(n)}
\cap\bigl[0,\Theta_{j-1}^{(n)}(r)\bigr]\bigr)\nonumber
\\
&&\qquad = \sum_{l=1}^{j-1}\sum
_{x_1=1}^{M_{1}} \sum_{i_1=1}^{N_r^{(x_1,1)}}
\sum_{i_{2}=2}^{\tilde G^{(x_1,i_1)}} \cdots\sum
_{i_{l}=2}^{\tilde G^{(x_1,i|_{l-1})}} \sum_{i_{l+2}=1}^{\tilde
{G}^{(x_1,i|_{l+1}^l)}}
\cdots
\\
&&\qquad\quad \sum_{i_{j}=1}^{\tilde
{G}^{(x_1,i|^l_{j-1})}}
\gamma^{(n)}_{j}\bigl(x_1,U^{(x_1,i|^l_2)},
\ldots,U^{(x_1,i|^l_j)}\bigr) T^{(x_1,i|^l_j)},\nonumber
\end{eqnarray}
where for $l=1$ the sum $\sum_{i_{2}=2}^{\tilde G^{(x_1,i_1)}}$ should
be absent in~(\ref{extranumber})--(\ref{extracont}); for $l=j-1$, the
expressions in~(\ref{eqnum})--(\ref{eqcont}) should be interpreted as
$1$ and
\[
\gamma^{(n)}_{j}\bigl(x_1,U^{(x_1,i|_2)},
\ldots,U^{(x_1,i|_{j-1},1)}\bigr) T^{(x_1,i|_{j-1},1)},
\]
respectively, and for $j=2$~(\ref{extranumber})--(\ref{extracont})
should be, respectively, interpreted as
\begin{eqnarray*}
K_2^{(n)}(r) &=& \sum
_{x_1=1}^{M_{1}}\sum_{i_1=1}^{N_r^{(x_1,1)}}1,
\\
\mu_2^{(n)}\bigl(\mathcal{E}_{2}^{(n)}
\cap\bigl[0,\Theta_{1}^{(n)}(r)\bigr]\bigr) &=&\sum
_{x_1=1}^{M_{1}}\sum_{i_1=1}^{N_r^{(x_1,1)}}
\gamma^{(n)}_{2}\bigl(x_1,U^{(x_1,i_1,1)}
\bigr) T^{(x_1,i_1,1)},
\end{eqnarray*}
%
from which we readily get 
\begin{eqnarray}
\label{extranumber2a}
E\bigl(K_2^{(n)}(r)\bigr) &=&
rM_{1},
\nonumber\\[-8pt]\\[-8pt]
E\bigl(\mu_2^{(n)}\bigl(
\mathcal{E}_{2}^{(n)} \cap\bigl[0,\Theta_{1}^{(n)}(r)
\bigr]\bigr)\bigr) &=&\frac{r}{M_2}\sum_{x|_2\in\mathcal{M}|_2}
\gamma^{(n)}_{2}(x|_2).\nonumber
\end{eqnarray}

For $j\geq3$, by conditioning on $N_r^{(x_1,1)},
G^{(x_1,i_1)},\ldots,G^{(x_1,i|_{j-2})}$, $U^{(x_1,i|_2)},
\ldots,\break  U^{(x_1,i|_{j-2})}$ (in the case of $j=3$,
$N_r^{(x_1,1)}, G^{(x_1,i_1)}$), and integrating on the
remaining\vadjust{\goodbreak} random variables, we get
from~(\ref{extranumber})
\begin{eqnarray*}
&& \sum_{l=1}^{j-1}
\sum_{x_1=1}^{M_{1}}
\cdots
\sum_{i_{l}=2}^{\tilde{G}{}^{(x_1,i|_{l-1})}}
\sum_{i_{l+2}=1}^{\tilde{G}^{(x_1,i|_{l+1}^l)}}\cdots
\\
&&\qquad \sum_{i_{j-1}=1}^{\tilde{G}^{(x_1,i|^l_{j-2})}}\frac{1}{M_{j-1}} \sum_{x_{j-1}=1}^{M_{j-1}}
\bigl(1 + M_{j} \gamma^{(n)}_{j-1}
\bigl(x_1,U^{(x_1,i|^l_2)},\ldots, U^{(x_1,i|^l_{j-2})},x_{j-1}
\bigr)\bigr).
\end{eqnarray*}
%
Proceeding inductively, we find
\begin{eqnarray}\label{esp1}
&& {E}\bigl(K_j^{(n)}(r)\bigr)\nonumber
\\[-2pt]
&&\qquad  =
\frac{r}{\prod_{p=1}^{j-2}M_{p+1}}
\sum_{l=1}^{j-1}\
\sum_{x|_{_{j-1}}\in\mathcal{M}|_{_{j-1}}}
\prod_{p=1}^{l-1} M_{p+1}
\gamma_{p}^{(n)}(x|_{p})
\\
&&\hspace*{160pt}{}\times  \prod
_{p = l + 1}^{j-1} \bigl(1 + M_{p+1}
\gamma_{p}^{(n)}(x|_{p})\bigr).\nonumber
\end{eqnarray}

Similarly,
\begin{eqnarray}\label{esp2}
\qquad && E\bigl(\mu_j^{(n)}\bigl(\mathcal{E}_{j}^{(n)}
\cap\bigl[0,\Theta_{j-1}^{(n)}(r)\bigr]\bigr)\bigr)\nonumber
\\
&&\qquad =\frac{r}{\prod_{p=1}^{j-1}M_{p+1}}\sum
_{l=1}^{j-1}\ \sum_{x|_{j}\in
\mathcal{M}|_{j}}
\gamma_{j}^{(n)}(x|_j) \prod
_{p=1}^{l-1} M_{p+1}\gamma_{p}^{(n)}(x|_{p})
\\
&&\hspace*{184pt}{}\times \prod_{p = l + 1}^{j-1} \bigl(1 +
M_{p+1}\gamma_{p}^{(n)}(x|_{p})
\bigr).\nonumber
\end{eqnarray}

We are now ready to argue our claims.
\begin{longlist}[(a)]
\item[(a)] Fix $j \in\{2,\ldots,k \}, r >0 $ and $L >0$. Then, using
    Jensen's inequality,
\begin{eqnarray}\label{exp1}
\nonumber
&&P \bigl(\min\bigl\{\xi_j^{(n)}(s)\dvtx  s \in
\mathcal{E}_j^{(n)} \cap\bigl[0,\Theta_{j-1}^{(n)}(r)
\bigr] \bigr\} > L \bigr)
\\
&&\qquad = \sum_{l=0}^{\infty} \biggl(1 -
\frac{L}{M_j} \biggr)^{l} P \bigl(K_{j}^{(n)}(r)=l
\bigr) = E \biggl[ \biggl(1 - \frac{L}{M_{j}} \biggr)^{K_{j}^{(n)}(r)}
\biggr]
\\
&&\qquad \geq \biggl(1 - \frac{L}{M_{j}} \biggr)^{E(K_{j}^{(n)}(r))}
= \biggl\{ \biggl(1 - \frac{L}{M_{j}} \biggr)^{M_{j}/L}
\biggr\}^{L
[E(K_{j}^{(n)}(r))/M_j ]}.\nonumber
\end{eqnarray}
Using~(\ref{esp1}), we find that the expression within square brackets
on the right-hand side of~(\ref{exp1}) is the expression
in~(\ref{em4}), which goes to 0 as $n\rightarrow\infty$ by hypothesis.
From~(\ref{em1}), we have that the expression within curly brackets on
the right-hand side of~(\ref{exp1}) is bounded away from zero as
$n\rightarrow\infty$. It immediately follows that the probability on
the left-hand side of~(\ref{exp1}) tends to 1 as $n\rightarrow\infty$,
and part (a) of Lemma~\ref{extramark} is established.

\item[(b)] Given $\varepsilon> 0$, by Markov's inequality,
\begin{equation}
\qquad P \bigl( \mu_{j}^{(n)} \bigl(\mathcal{E}_j^{(n)}
\cap\bigl[0, \Theta_{j-1}^{(n)}(r)\bigr] \bigr) > \varepsilon
\bigr) \leq\varepsilon^{-1} {E} \bigl(\mu_{j}^{(n)}
\bigl(\mathcal{E}_j^{(n)} \cap\bigl[0, \Theta_{j-1}^{(n)}(r)
\bigr] \bigr) \bigr)
\end{equation}
\end{longlist}%
and the result follows from~(\ref{esp2}) and~(\ref{em3}).
\end{pf}

%
\begin{rmk}\label{continuity}
If $X_{k-1}^{(n)} \rightarrow X_{k-1}$ a.s. as $n \rightarrow\infty$ in
Skorohod space, then, by Proposition 5.2 in Chapter~3 of~\cite{knek}
(page 118), we have that
\begin{equation}
\lim_{n \rightarrow\infty} X_{k-1}^{(n)}(s) = \lim
_{n \rightarrow\infty} X_{k-1}^{(n)}(s-) =
X_{k-1}(s)
\end{equation}
for all $s \geq0 $ which is a continuity point of $X_{k-1}$.
\end{rmk}

%
\begin{Lem} \label{clock}
Assume that the induction hypotheses~(\ref{convxj})--(\ref{convgj})
hold, and let $r \in[0, \infty)$ be fixed.
Then
$\Gamma_{k}^{(n)}(r) \rightarrow\Gamma_{k}(r)$
in probability as $n \rightarrow\infty$.
\end{Lem}

\begin{pf}
The strategy is to separate the contribution of the extra marks and
Poissonian marks with large labels from the remaining contributions.
The convergence of the remaining main (as it turns out) contributions
to the corresponding infinite volume contributions follows readily from
the first part of~(\ref{em2}), since there is only a fixed finite
number of contributions involved. The negligibility of the total
contribution of extra marks was established in Lemma~\ref{extramark}
above, so we are left with establishing that of the total contribution
of high label marks. Details follow.

Let $\Psi_{k}^{(n)}(r)=\mu_{k}^{(n)} ((\mathcal
{S}_{k}^{(n)}\setminus\mathcal{E}_{k}^{(n)})\cap[0,r] )$.
Then
\begin{equation}
\label{eqbr} \bigl|\Gamma_{k}(r) - \Gamma_{k}^{(n)}(r)\bigr|
\leq\bigl|\Gamma_{k}(r) - \Psi_{k}^{(n)}(r)\bigr| +
\mu_{k}^{(n)}\bigl(\mathcal{E}_{k}^{(n)}
\cap[0,r]\bigr).
\end{equation}

By Lemma~\ref{extramark}(b), the second term on the right
of~(\ref{eqbr}) goes to 0 in probability as $n \rightarrow\infty$. We
will argue that so does the first one. In preparation for this, let us
take, for given $\varepsilon> 0$, $m_1 \in\mathbb{N}_*$ such that
\begin{equation}
\label{m1} \sum_{x_1 > m_1} \sum
_{x|^2 \in\mathbb{N}_*^{k-1}}\bar\gamma_{k}(x|_k) \leq
\varepsilon/k
\end{equation}
[recall the notation introduced around~(\ref{eqbm}) above]. Proceeding
inductively, with $m_1, \ldots,m_{j-1}$, $2\leq j\leq k-1$, fixed,
choose $m_j$ such that
\begin{equation}
\label{mj} \sum_{x_1=1}^{m_1}\cdots\sum
_{x_{j-1}=1}^{m_{j-1}}\ \sum
_{x_{j}>m_j}\sum_{x|^{j+1} \in\mathbb{N}_*^{j+1}}^{\infty}
\bar\gamma_{k}(x|_k)\leq\varepsilon/k
\end{equation}
and with $m_1, \ldots,m_{k-1}$ fixed, choose $m_k$ such that 
\begin{equation}
\label{mk} \sum_{x_1=1}^{m_1}\cdots\sum
_{x_{k-1}=1}^{m_{k-1}}\sum
_{x_k >
m_k}^{\infty}\bar\gamma_{k}(x|_k)
\leq\varepsilon/k.
\end{equation}
This procedure is well defined by~(\ref{eqfin}).


Going back to the first term on the right of~(\ref{eqbr}), we have that
\begin{eqnarray}
\qquad &&\bigl|\Gamma_{k}(r) - \Psi_{k}^{(n)}(r)\bigr|\nonumber
\\
\label{rs1} &&\qquad \leq\biggl\llvert\sum_{s \in\tilde{\mathcal{S}}_k^{(m_k)}
\cap[0,r]} \bigl\{
\gamma_k\bigl(X_{k-1}(s), \xi_k(s)\bigr)-\gamma
_k^{(n)}\bigl(X_{k-1}^{(n)}(s),
\xi_k(s)\bigr) \bigr\} T_s^{(k)}\biggr\rrvert
\\
\label{rs2} &&\quad\qquad{} +\sum_{s\in(\mathcal{S}_k\setminus\tilde
{\mathcal{S}}_k^{(m_k)})\cap[0,r]}\gamma_k
\bigl(X_{k-1}(s),\xi_k(s)\bigr) T_s^{(k)}
\\
\label{rs3} &&\quad\qquad{} +\sum_{s\in(\mathcal{S}_k^{(n)}\setminus\tilde
{\mathcal{S}}_k^{(m_k)})\cap[0,r]} \gamma_k^{(n)}
\bigl(X_{k-1}^{(n)}(s),\xi_k(s)\bigr)
T_s^{(k)}
\end{eqnarray}
[recall~(\ref{tsj})]. We have used here 
the fact that, given the coupled construction
of $X^{(n)}_k$ and $X_k$, we have that $\xi^{(n)}_k(s)=\xi_k(s)$ for
Poisson points $s$.

The expression on the right-hand side of~(\ref{rs1}) converges to 0 in
probability as $n$ increases because it is a finite sum, and from the
first part of~(\ref{em2}), and since $X_{k-1}^{(n)}(s) = X_{k-1}(s)$
for all $s \in\tilde{\mathcal{S}}_k^{(m_{k})} \cap[0,r]$ for all large
enough $n$ almost surely, as follows from Remark~\ref{continuity}
above, and the fact that the points of
$\tilde{\mathcal{S}}_k^{(m_{k})}$ are almost surely continuity ponts of
$X_{k-1}$.

Let $B$ and $C$ denote the expressions in~(\ref{rs2}) and~(\ref{rs3}),
respectively.

To analyze $B$, we start by taking, for given $\eta> 0$, $r_0'$ such that
\begin{equation}
P\bigl(\Theta_{k-1}\bigl(r'_0\bigr) > r\bigr)
\geq1 - \eta.
\end{equation}
This is allowed by the second assertion of Lemma~\ref{fin}. Now letting
\[
\mathcal{A}_0 = \bigl(\mathcal{S}_k \setminus\tilde{\mathcal{
S}}_k^{(m_k)}\bigr) \cap\bigl[0,\Theta_{k-1}
\bigl(r_0'\bigr)\bigr],
\]
we define
\begin{eqnarray}
\mathcal{A}_1 &=& \bigl\{s \in\mathcal{A}_0\dvtx
X_{k-1,1}(s) > m_1\bigr\},\nonumber
\\
\mathcal{A}_2 &=& \bigl\{s \in\mathcal{A}_0\dvtx
X_{k-1,1}(s) \leq m_1, X_{k-1,2}(s) >
m_2\bigr\},\nonumber
\\
&&\hspace*{95pt}\vdots \nonumber\\[-8pt]\\[-8pt]
\mathcal{A}_{k-1} &=& \bigl\{s \in\mathcal{A}_0\dvtx
X_{k-1,1}(s) \leq m_1, \ldots,X_{k-1,k-2}(s)\leq
m_{k-2},\nonumber
\\
&&\hspace*{136pt}X_{k-1,k-1}(s)> m_{k-1} \bigr\},\nonumber
\\
\mathcal{A}_{k} &=& \bigl\{s \in\mathcal{A}_0\dvtx
X_{k-1,1}(s) \leq m_1, \ldots,X_{k-1,k-1}(s) \leq
m_{k-1} \bigr\}.\nonumber
\end{eqnarray}
Notice that by the definition of $\mathcal{A}_0$ and $\tilde
{\mathcal{S}}_k^{(m_k)}$ (recall~(\ref{tsj}) above), we have that
$\xi_k(s)>m_k$. Then, outside an event of probability at most $\eta$,
we have that
\begin{equation}
B \leq\sum_{i = 1}^{k} \sum
_{s \in\mathcal{A}_i} \gamma_{k}\bigl(X_{k-1}(s),
\xi_k(s)\bigr) T_s^{(k)}
\end{equation}
and following the same arguments used to establish~(\ref{etheta}), and
using\break  \mbox{(\ref{m1})--(\ref{mk}),} we conclude that
\begin{eqnarray*}
&& E\Biggl( \sum_{i = 1}^{k} \sum
_{s \in\mathcal{A}_i} \gamma_{k}\bigl(X_{k-1}(s),
\xi_k(s)\bigr) T_s^{(k)}\Biggr)
\\
&&\qquad  \leq r_0'\Biggl[\sum_{x_1 > m_1} \sum
_{x|^2 \in\mathbb{N}_*^{k-1}} \bar\gamma_{k}(x|_k)
+\sum_{x_1=1}^{m_1} \sum
_{x_2 > m_2}\sum_{x|^3\in\mathbb
{N}_*^{k-2}} \bar
\gamma_{k}(x|_k)
\\
&&\hspace*{111pt}{} +\cdots+\sum_{x_1=1}^{m_1}
\cdots\sum_{x_{k-1}=1}^{m_{k-1}}\sum
_{x_k > m_k} \bar\gamma_{k}(x|_k)\Biggr]
\leq r_0' \varepsilon,
\end{eqnarray*}
where the first inequality comes from ignoring the restriction $s\in
\mathcal{A}_0$ in the first $k-1$ terms
of the sum in $i$.
This shows that $B \rightarrow0$ in probability as $\epsilon\to0$,
since $\eta$ is arbitrary.

The analysis of $C $ is similar, with the dependence on $n$ as a
distinctive aspect.
From induction hypothesis~(\ref{convgj}) and the second assertion of
Lemma~\ref{fin}, given $\eta> 0$, there exists $r_0$ such that for all
$n$ sufficiently large,
\begin{equation}
P\bigl(\Theta_{k-1}^{(n)}(r_0) > r\bigr) \geq1 -
\eta;
\end{equation}
recall~(\ref{then}). With such $r_0$ and the above choice of
$m_1,\ldots,m_k$, define
\begin{eqnarray}
\mathcal{A}_0^{(n)}&=& \bigl(\mathcal{S}_k^{(n)}
\setminus\mathcal{S}_k^{(m_k)} \bigr)\cap\bigl[0,
\Theta_{k-1}^{(n)}(r_0)\bigr],\nonumber
\\
\mathcal{A}_1^{(n)} &=& \bigl\{s \in\mathcal{A}_0^{(n)}\dvtx
X_{k-1,1}^{(n)}(s) > m_1\bigr\},\nonumber
\\
\mathcal{A}_2^{(n)} &=& \bigl\{s \in\mathcal{A}_0^{(n)}\dvtx
X_{k-1,1}^{(n)}(s) \leq m_1, X_{k-1,2}^{(n)}(s)
> m_2\bigr\},
\\
&&\hspace*{102pt}\vdots\nonumber
\\
\mathcal{A}_k^{(n)}&=&\bigl\{s\in\mathcal{A}_0^{(n)}\dvtx
X_{k-1,1}^{(n)}(s)\leq m_1,\ldots,X_{k-1,k-1}^{(n)}(s)
\leq m_{k-1} \bigr\}.\nonumber
\end{eqnarray}
Then
\begin{equation}
P \Biggl(C\leq\sum_{i=1}^{k-1}\sum
_{s\in\mathcal{A}_i^{(n)}}\gamma_{k}^{(n)}
\bigl(X_{k-1}^{(n)}(s),\xi_k(s)\bigr)T_s^{(k)}
\Biggr) 
\geq1-\eta
\end{equation}
for all $n$ large enough. By~(\ref{em2}), we may take $n$ suficiently
large such that
\begin{eqnarray}\label{mkn}
\Biggl\llvert\sum_{x_1 > m_1}^{M_1}
\sum_{x|^2\in\mathcal
{M}|^2}\bar{\gamma}_{k}^{(n)}(x|_k)
- \sum_{x_1 > m_1} \sum_{x|^2\in\mathbb{N}_*^{k-1}}
\bar{\gamma}_{k}(x|_k)\Biggr\rrvert &\leq&\varepsilon/ k,\hspace*{-36pt}\nonumber
\\
\Biggl\llvert\sum_{x_1 = 1}^{m_1}
\sum_{x_2=m_2+1}^{M_2}\sum
_{x|^3\in\mathcal{M}|^3}\bar{\gamma}_{k}^{(n)}(x|_k)
-\sum_{x_1=1}^{m_1}
\sum_{x_2>m_2}\sum_{x|^3\in\mathbb
{N}_*^{k-2}}
\bar{\gamma}_{k}(x|_k)\Biggr\rrvert&\leq&\varepsilon/k,\hspace*{-36pt}
\\
&&\hspace*{-140pt}\vdots \nonumber
\\
\Biggl\llvert\sum
_{x_1=1}^{m_1}\cdots\sum_{x_{k-1}=1}^{m_{k-1}}
\sum_{x_k = m_k+1}^{M_k} \bar{\gamma}_{k}^{(n)}(x|_k)
- \sum_{x_1=1}^{m_1}\cdots\sum
_{x_{k-1}=1}^{m_{k-1}} \sum_{x_k > m_k}
\bar{\gamma}_{k}(x|_k)\Biggr\rrvert&\leq&\varepsilon/ k.\hspace*{-36pt}\nonumber
\end{eqnarray}
Following the same arguments used to establish~(\ref{etheta}), 
and using~(\ref{m1})--(\ref{mk}) and~(\ref{mkn}), we get
\begin{eqnarray*}
&& E\Biggl( \sum_{i = 1}^{k-1} \sum
_{s \in\mathcal{A}_i^{(n)}} \gamma_{k}^{(n)}
\bigl(X_{k-1}^{(n)}(s),\xi_k(s)\bigr)
T_s^{(k)}\Biggr)
\\
&&\qquad \leq r_0'
\Biggl[\sum_{x_1 = m_1+1}^{M_1}\sum
_{x|^2 \in\mathcal
{M}|^2} \bar\gamma_{k}^{(n)}(x|_k)
+ \sum_{x_1=1}^{m_1} \sum
_{x_2 = m_2+1}^{M_2} \sum_{x|^3\in\mathcal{M}|^3}
\bar\gamma_{k}^{(n)}(x|_k)
\\
&&\hspace*{126pt}{} + \cdots+\sum
_{x_1=1}^{m_1} \cdots\sum_{x_{k-1}=1}^{m_{k-1}}
\sum_{x_k = m_k+1}^{M_k} \bar\gamma_{k}^{(n)}(x|_k)
\Biggr]
\\
&&\qquad \leq\bigl(r_0+r_0'\bigr) \varepsilon.
\end{eqnarray*}
This shows that $C \rightarrow0$ in probability as we first take $n
\rightarrow\infty$ and then $\epsilon\to0$, since $\eta$ is arbitrary,
thus completing the proof.
\end{pf}

%
\begin{cor} \label{cor}
The result of Lemma~\ref{clock} still holds if we replace $r$ on the
left-hand side by $r_n$ with $r_n\to r$ as $n\to\infty$, with $(r_n)$ a
deterministic sequence.
\end{cor}

\begin{pf}
Let us write
\begin{equation}
\label{cor1} \bigl|\Gamma_k^{({n})}(r_n)-
\Gamma_k(r)\bigr|\leq\bigl|\Gamma_k^{({n})}(r_n)-
\Gamma_k^{({n})}(r)\bigr|+\bigl|\Gamma_k^{({n})}(r)-
\Gamma_k(r)\bigr|.
\end{equation}
Using the hypothesis and the monotonicity of $\Gamma_k^{({n})}$, given
$\delta>0$, we have that the first term on the right-hand side
of~(\ref{cor1}) is bounded above by $\Gamma_k^{({n})}(r+\delta
)-\Gamma_k^{({n})}(r-\delta)$ for all $n$ large enough, which is in
turn bounded above by
\begin{eqnarray}
\label{cor2}
&&\Gamma_k(r+\delta)-\Gamma_k(r-\delta)+\bigl|
\Gamma_k^{({n})}(r+\delta)-\Gamma_k(r+\delta)\bigr|
\nonumber\\[-8pt]\\[-8pt]
&&\qquad{} +\bigl|
\Gamma_k^{({n})}(r-\delta)-\Gamma_k(r-\delta)\bigr|.\nonumber
\end{eqnarray}

Let $\eta>0$ now be given. By Lemma~\ref{clock}, and
using~(\ref{cor1})--(\ref{cor2}), we find that
\begin{equation}
\label{cor3} \qquad\quad\limsup_{n\to\infty}P\bigl(\bigl|\Gamma_k^{({n})}(r_n)-
\Gamma_k(r)\bigr|>\eta\bigr)\leq P\bigl(\Gamma_k(r+\delta)-
\Gamma_k(r-\delta)>\eta/3\bigr)
\end{equation}
and the result follows from $r$ being almost surely a continuity point
of $\Gamma_k$ (see Remarks~\ref{atom} and~\ref{atomb}), since $\delta$
is arbitrary.
\end{pf}

%
\begin{rmk}\label{rmktheta}
The same argument, of course, works in the case when $(r_n,r)$ are
random and independent of $(\Gamma^{({n})}_k,\Gamma_k)$ and $r_n\to r$
almost surely as $n\to\infty$. This can be applied to establish that
under the assumption of Lemma~\ref{clock}, we have that
\begin{equation}
\label{eqtheta} \Theta_k^{({n})}(T) \rightarrow
\Theta_k(T)\qquad\mbox{as } {n} \rightarrow\infty,
\end{equation}
in probability for every $T\geq0$.
\end{rmk}

The next result is a finite volume version of Lemma~\ref{hypercube} in
the above section.

%
\begin{Lem} \label{hypercubeprob}
Given $k\geq2$, $n\geq$1, let $X_{k}^{(n)}$ be a trap model on
${\mathbb T}^{(n)}_k$. Suppose that the assumptions of Theorem~\ref{GK}
and the induction hypotheses (\ref{convxj})--(\ref{convgj}) all hold.
Then given $m\geq1$ and $\epsilon>0$, there exists $\tilde m=\tilde
m^{(k)}\geq1$ such that the event
\begin{eqnarray}
\nonumber
A_{m}^{(n)} &=& \bigl\{\mbox{if
}X_{k,i}^{(n)}(t) > \tilde m\mbox{ for some }i=1,\ldots,k-1
\mbox{ and }t \in[0,T],
\nonumber\\[-8pt]\\[-8pt]
&&\hspace*{83pt}\mbox{then }X_{k,j}^{(n)}(t) > m
\mbox{ for }j = i+1,\ldots,k\bigr\}\nonumber
\end{eqnarray}
has probability bounded below by $1-\epsilon$ for all $n$ large enough.
\end{Lem}

\begin{pf}
We argue similarly as in the proof of Lemma~\ref{hypercube}, except
that statements here hold \textit{with high probability}, rather than
almost surely.

By Remark~\ref{rmktheta} and the fact that
$\lim_{r\to\infty}\Theta_{j}(r)=\infty$ for every $1\leq j\leq k$, we
may choose $T' > 0$ such that $\Theta^{(n)}_{k}(T') > T$ with
probability at least $1-\epsilon/4$ uniformly in $n$. Now for $m\in
\mathbb{N}_{*}$ and $j=1,\ldots,k$, let $\tilde{\mathcal{S}}_{j}^{(m)}$
be\vspace*{-1pt} as in~(\ref{tsj}) above. For fixed
$\ell\in\mathbb{N}_{*}$, let $\mathcal{T}^{(n)}_{kj}(\ell)$ denote the
set of times up to $\Theta^{(n)}_{k-1}(T')$ spent by $X^{(n)}_{k-1,j}$
above $\ell$, $j=1,\ldots,k-1$. Analogously as for the infinite volume
case [see~(\ref{tt1})], we may check that the expected Lebesgue measure
of $\mathcal{T}^{(n)}_{kj}(\ell)$ equals
\begin{equation}
\label{tt1n} T'\sum_{x_{1}=1}^{M_1}
\cdots\sum_{x_j =\ell+1}^{M_j}\cdots\sum
_{x_{k-1}=1}^{M_{k-1}} \gamma^{(n)}_{1}(x|_1)
\cdots\gamma^{(n)}_{k-1}(x|_{k-1})
\end{equation}
plus the contribution of the extra marks and their descendants. It
follows from~(\ref{em2}) that the $\limsup_{n\to\infty}$ of the
expression in~(\ref{tt1n}) vanishes as $\ell\to\infty$. By
Lemma~\ref{extramark}(b), the contribution of the extra marks and their
descendants vanishes in probability as $n\to\infty$. It then follows
from elementary properties of Poisson processes, that
\begin{equation}
\label{tt2n} \Biggl\{\bigcup_{j=1}^{k-1}
\mathcal{T}^{(n)}_{kj}(\ell)\Biggr\}\cap\tilde{\mathcal{
S}}_{k}^{(m)}=\varnothing
\end{equation}
outside an event whose probability is bounded above by $\epsilon/4$ for
all $\ell, n$ large enough. This statement is about Poissonian marks;
but it also holds for extra marks by Lemma~\ref{extramark}(a).

So, given $m\in\mathbb{N}_{*}$ and $\epsilon>0$, we find $\hat m^{(k)}$
such that outside an event of probability smaller than $\epsilon/2$ for
all $n$ large enough,\vspace*{-1pt} if $X^{(n)}_{k,i}(t) >\hat m^{(k)}$
for any $t\leq T$, then $X^{(n)}_{k,k}(t) > m$. This in particular
establishes the claim for $k=2$ by the choice $\tilde m^{(2)}=\hat
m^{(2)}$. Let us assume that the claim is established for $k-1$.
Then substituting in that claim $\epsilon$ for $\epsilon/4$, and
$T$ for $T''$
such that $P(\Theta^{(n)}_{k-1}(T')\leq T'')>1-\epsilon/4$ for all
large enough $n$ as $T$,
and choosing $\tilde m^{(k)}=\tilde m^{(k-1)}\vee\hat m^{(k)}$, 
we find that it satisfies the claim for $k$.
\end{pf}




\subsection{\texorpdfstring{End of proof of Theorem \protect\ref{GK}}
{End of proof of Theorem 5.4}}\label{pf}

For $j = 1,\ldots,k$, $n\geq1$, let
$X_{j}^{(n)} \sim TM({\mathbb T}_{k}^{(n)},\underline{\gamma_{j}^{(n)}})$
and $X_{j} \sim K_{j}({\mathbb T}_k,\underline{\gamma_{j}})$, and, for
$k \geq2$ fixed, suppose that $X_{k-1}^{(n)} \rightarrow X_{k-1}$ in
probability as
$n \rightarrow\infty$.
We may then and will inductively suppose that
\begin{equation}
\label{convxk-1} X_{k-1}^{(n')} \rightarrow X_{k-1}\qquad
\mbox{a.s. as } n' \rightarrow\infty,
\end{equation}
%
for a subsequence $(n')$.
We will fix $\epsilon>0$, $T>0$ and $m\geq1$ and choose $T'$ and
$\tilde m$
such that outside an event $\mathcal{E}=\mathcal{E}_{n'}$ of
probability at most $\epsilon/2$ for all $n'$ large enough, we have that
the conclusions of
Lemma~\ref{hypercube} and~\ref{hypercubeprob} hold, and also that
$\Theta_{k-1}(T')\wedge\Theta^{(n')}_{k-1}(T')>T$. We will also assume
that $\tilde m\geq m$, and that the claims of Lemma~\ref{extramark}
hold almost surely over $(n')$.


On the way to showing the validity of~(\ref{convxk-1}) with $k$
replacing $k-1$ (in probability), we now proceed to define appropriate
time distortions $\lambda^{(n')}$; see the discussion on the Skorohod
metric at the beginning of Section~\ref{infvol}.
Let us start by considering the constancy intervals of $X_{k-1,1}$ in
$[0,\Theta_{k-1}(T'))$ with $X_{k-1,1}\leq\tilde m$. These are defined
to be the \textit{rank-$\tilde m$ constancy intervals of the level $1$
for $X_{k-1}$}. Proceeding inductively, given $2\leq\ell\leq k-1$, for
each rank-$\tilde m$ constancy interval $I$ of level $\ell-1$, we
consider the constancy intervals of $X_{k-1,\ell}$ inside $I$ such that
$X_{k-1,\ell}\leq\tilde m$. The collection of all such intervals
obtained from all the rank-$\tilde m$ constancy intervals of level
$\ell-1$ for $X_{k-1}$ form the set of rank-$\tilde m$ constancy
intervals of level $\ell$ for $X_{k-1}$.

Let $a_1,\ldots,a_{2L}$ denote the collection of all endpoints of all
the rank-$\tilde m$ constancy intervals of level $i$ for $X_{k-1}$,
$i=1,\ldots,k-1$, in increasing order, and let $b_1,\ldots,b_{2J}$
denote the collection of all endpoints of all the rank-$\tilde m$
constancy intervals of level $k-1$ for $X_{k-1}$ in increasing order.
See Figure~\ref{figproof}.

%
\begin{figure}

\includegraphics{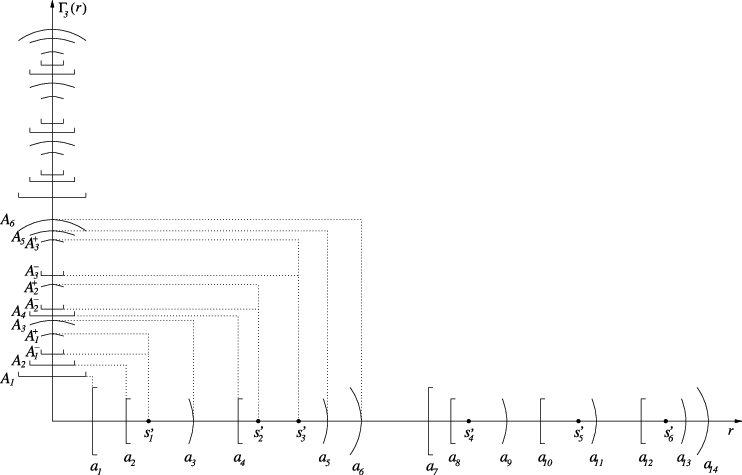}

\caption{Depiction of objects appearing in the argument for the proof
of Theorem~\protect\ref{GK} with $k=3$. Rank-$\tilde m$ constancy
intervals of the
first level ($[a_1,a_6)$ and $[a_7,a_{14})$) and rank-$\tilde m$
constancy intervals of the second level
($[a_2,a_3)$, $[a_4,a_5)$, $[a_8,a_9)$, $[a_{10},a_{11})$ and
$[a_{12},a_{13})$) for $X_{2}$ appear on
the $x$-axis. Correspondingly on the $y$-axis, we have rank-$\tilde m$
constancy intervals of the first level ($[A_1,A_6)$ and another
whose endpoints are not named in the picture),
rank-$\tilde m$ constancy intervals of the second level
($[A_2,A_3)$, $[A_4,A_5)$, and others whose endpoints are not named in
the picture),
and rank-$\tilde m$ constancy intervals of the third level
($[A_1^-,A_1^+)$, $[A_2^-,A_2^+)$, $[A_3^-,A_3^+)$, and others whose
endpoints are not named in the picture)
for $X_{3}$. Some of the correspondences between the axes are indicated
by dotted lines.
We have also $b_1=a_2$, $b_2=a_3$, $b_3=a_4$, $b_4=a_5$, $b_5=a_8$,
$b_6=a_9$, $b_7=a_{10}$, $b_8=a_{11}$, $b_9=a_{12}$ and
$b_{10}=a_{13}$. This picture is also good for $X^{(n)}_{3}$, with $n$
large, and with all endpoint labels having superscripts
``$(n)$.''}\label{figproof}
\end{figure}

Let us also consider rank-$\tilde m$ constancy intervals of level $i$
for $X^{(n')}_{k-1}$, with the paralell definition to the one above. By
the assumption that Lemma~\ref{extramark} holds almost surely over
$(n')$, and for $n'$ large enough, there is one-to-one correspondence
of the $a^{(n')}_1,\ldots,a^{(n')}_{2L^{(n')}}$ and
$a_1,\ldots,a_{2L}$, with $L^{(n')}=L$ for all large $n'$ and
$a^{(n')}_i$ corresponding to $a_i$, and from~(\ref{convxk-1}),
\begin{equation}
\label{antoa} a^{(n')}_i\to a_i
\end{equation}
almost surely as $n'\to\infty$ for every $i=1,\ldots,2L$.

Let now $A_i=\Gamma_k(a_i)$ and $A^{(n')}_i=\Gamma
^{(n')}_k(a^{(n')}_i)$, $i=1,\ldots,2L$. See Figure~\ref{figproof}. It
follows from Lemma~\ref{clock} (see Corollary~\ref{cor} and
Remark~\ref{rmktheta}) that
\begin{equation}
\label{AntoA} A^{(n')}_i\to A_i
\end{equation}
in probability as $n'\to\infty$, and we may assume a.s. convergence by
taking a subsequence.

Let $\{s'_1,\ldots,s'_Q\}$ be the enumeration in increasing order of
$\bigcup_{i=1}^J\{\tilde{\mathcal{S}}_{k}^{(m)}\cap[b_{2i-1},b_{2i})\}$,
and let $A_i^-=\Gamma_k(s'_i-)$ and $A^+_i=\Gamma_k(s'_i)$. See
Figure~\ref{figproof}. We note that the intervals $[A_i^-,A^+_i)$,
$i=1,\ldots,Q$ are the rank-$m$ constancy intervals of level $k$ for
$X_{k}$, whereas $A_i$, $i=1,\ldots,2L$, are the endpoints of all the
rank-$\tilde m$ constancy intervals of level $i$ for $X_{k}$,
$i=1,\ldots,k-1$.

We remark at this point that under our assumptions so far, we have that
for all $i=1,\ldots,J$
\begin{equation}
\label{ids} \tilde{\mathcal{S}}_{k}^{(m)}
\cap\bigl[b_{2i-1},b_{2i}\bigr)=\tilde{\mathcal{S}}_{k}^{(m)}
\cap\bigl[b^{(n')}_{2i-1},b^{(n')}_{2i}\bigr)
\end{equation}
almost surely for all $n'$ large enough, where
$b^{(n')}_{i}=a^{(n')}_{j}$ such that $a_{j}=b_{i}$.

Let $A_i^-(n')=\Gamma^{(n')}_k(s'_i-)$, $A^+_i(n')=\Gamma
^{(n')}_k(s'_i)$. Then we have that for all large enough $n'$,
$[A_i^-(n'),A^+_i(n'))$, $i=1,\ldots,Q$ are the rank-$m$ constancy
intervals of level $k$ for $X^{(n')}_{k}$, whereas $A^{(n')}_i$,
$i=1,\ldots,2L$, are\vspace*{-1pt} the endpoints of all the
rank-$\tilde m$ constancy intervals of level $i$ for $X^{(n')}_{k}$,
$i=1,\ldots,k-1$.

Let us now argue that
\begin{equation}
\label{ApmntoApm} A_i^\pm{\bigl(n'\bigr)}
\to A^\pm_i
\end{equation}
in probability as $n'\to\infty$ (and again we may assume a.s.
convergence by taking a subsequence). It is enough to first note that
almost surely $A_i^-(n')= \tilde\Gamma_{k}^{(n')}(s'_i)$,
$A_i^-=\tilde\Gamma_{k}(s'_i)$,
$A_i^+(n')=\tilde\Gamma_{k}^{(n')}(s'_i)+\gamma
^{(n)}_k(X^{(n')}_{k-1}(s'_i),\xi_k(s'_i))$ and
$A_i^+=\tilde\Gamma_{k}(s'_i)+\gamma_k(X_{k-1}(s'_i),\xi_k(s'_i))$,
where $\tilde\Gamma_{k}^{(n')}$ and $\tilde\Gamma_{k}$ are obtained
from $\tilde\mu_{k}^{(n')}$ and $\tilde\mu_{k}$ as $\Gamma_{k}^{(n')}$
and $\Gamma_{k}^{(n')}$ are obtained from $\mu_{k}^{(n')}$ and
$\mu_{k}$, respectively, where $\tilde\mu_{k}^{(n')}=\mu_{k}^{(n')}$
and $\tilde\mu_{k}=\mu_{k}$ everywhere except at $\{s'_i\}$, where
$\tilde\mu_{k}^{(n')}$ and $\tilde\mu_{k}$ both vanish. By the same
arguments above we get
$\tilde\Gamma_{k}^{(n')}(s'_i)\to\tilde\Gamma_{k}(s'_i)$ in
probability, and the result follows upon noticing that
$X^{(n')}_{k-1}(s'_i)=X_{k-1}(s'_i)$ for all large enough $n'$ and
using~(\ref{em2}).

We are now ready to define our time distortion. Let $\lambda
^{(n')}\dvtx [0,\infty)\to[0,\infty)$ be such that
\begin{equation}
\label{id} \qquad\lambda^{(n')}(A_i)=A_i^{(n')},
\qquad\lambda^{(n')}\bigl(A^-_i\bigr)=A_i^-
\bigl(n'\bigr),\qquad\lambda^{(n')}\bigl(A^+_i
\bigr)=A_i^+\bigl(n'\bigr)
\end{equation}
and make\vspace*{-2pt} it linear between successive points of
$\mathcal{A}:=\{A_i, i=1,\ldots,2L;\break A^-_j,  A^+_j,
j=1,\ldots,Q\}$, and linear with inclination 1 from $\max\mathcal{A}$
on. Then $\lambda^{(n')}$ is almost surely well defined for all large
enough $n'$, and one readily checks that condition~(\ref{id}) implies
that $\lambda^{(n')}$ maps rank-$\tilde m$ constancy intervals of level
$i$ for $X_{k}$ to the corresponding\vspace*{-1pt} rank-$\tilde m$
constancy intervals of level $i$ for $X^{(n')}_{k}$, $i=1,\ldots,k$,
given by the coupling. In particular,
$X_{k,j}(\lambda^{(n')}(\cdot))=X^{(n')}_{k,j}(\cdot)$, $j=1,\ldots i$,
on those respective intervals. From the assumptions of the paragraph
of~(\ref{convxk-1}), we then have that outside $\mathcal E$
\begin{equation}
\label{dist} \sup_{0\leq u\leq T}\rho\bigl(X^{(n')}_k,X_k,
\lambda^{(n')},u\bigr) \leq1/m
\end{equation}
and by~(\ref{AntoA}) and~(\ref{ApmntoApm}) and our construction and
assumptions it follows that
\begin{equation}
\label{phito0} \phi\bigl(\lambda^{(n')}\bigr)\to0
\end{equation}
as $n'\to\infty$ almost surely, where $\phi$ is the time distortion
function introduced in~(\ref{eqg}).

It follows from all of the above that for every fixed $\epsilon,T>0$
and $m\in\mathbb{N}_{*}$ we may find a subsequence $(n')$ such that
\begin{equation}
\label{limsup} P \biggl(\rho\bigl(X^{(n')}_k,X_k
\bigr)>\frac{1}m +e^{-T} \biggr)\leq\epsilon
\end{equation}
for all $n'$ large enough, so we have that $X^{(n')}_k\to X_k$ in
probability, and this
readily implies the claim of Theorem~\ref{GK}.


\subsection{\texorpdfstring{Proof of Theorem \protect\ref{scal}}
{Proof of Theorem 5.2}}\label{GREM}

The strategy will be to work with a coupled version of $(\underline
{\gamma_{k}^{(n)}},\underline{\gamma_{k}})$,
which we will call $(\underline{\hat\gamma_{k}^{(n)}},\underline
{\hat\gamma_{k}})$, such that
almost surely for every $j = 1,\ldots,k$ and $x|_{j}\in\mathbb{N}_{*}^{j}$
\begin{equation}
\label{eqgntohg} \hat\gamma_{j}^{(n)}(x|_{j})\to
\hat\gamma_{j}(x|_{j})\qquad\mbox{as } n\to\infty
\end{equation}
and then verify the remaining conditions of Theorem~\ref{GK}. (Recall
that in the context of Theorem~\ref{scal}, the sets of parameters
$\underline{\gamma_{k}^{(n)}}$ and $\underline{\gamma_{k}}$ are
random.)

For the coupling, we use the construction of~\cite{kng}, Section~6,
which we describe briefly, guiding the reader to that reference for
more details.

Let $E_j(x|_j)$, $x|_j\in\mathbb{N}_{*}^{j}$, $j=1,\ldots,k$ be
independent mean one exponential random variables,
and, for $x|_{j-1}\in\mathbb{N}_{*}^{j-1}$ make
\begin{equation}
\label{eqS} S_j(x|_{j})=\sum
_{i=1}^{x_j}E_j(x|_{j-1},i),
\end{equation}
where $x|_0$ is a void symbol. Let now
\begin{eqnarray}
\label{eqhgn} \hat\gamma_{j}^{(n)}(x|_{j})&=&c_{j}^{(n)}
G_j^{-1} \biggl(\frac
{S_j(x|_{j})}{S_j(x|_{j-1},M_j+1)} \biggr),
\\
\label{eqhg} \hat\gamma_{j}(x|_{j})&=&S_j(x|_{j})^{-1/\alpha_j}.
\end{eqnarray}
From an elementary large deviation estimate, we may assume that
\begin{equation}
\label{eqld} S_j(x|_{j-1},M_j+1)
\leq2M_j
\end{equation}
for all $x|_{j-1}\in\mathcal{{M}}|_{j-1}$ and $n$ sufficiently large
almost surely,
recalling that the $M_j$'s depend on $n$.

Then $\hat\gamma_{j}^{(n)}(x|_{j-1})$ and $\hat\gamma_{j}(x|_{j-1})$
are\vspace*{1pt} versions of $\gamma_{j}^{(n)}(x|_{j-1})$ and
$\gamma_{j}(x|_{j-1})$, respectively~\cite{knlwz}. Proposition 6.3 and
Lemma 6.4 of~\cite{kng} immediately imply~(\ref{eqgntohg}), and also
that almost surely for every $j = 1,\ldots,k$ and
$x|_{j-1}\in\mathbb{N}_{*}^{j-1}$
\begin{equation}
\label{eqhgntohg} \sum_{x_{j}\in\mathcal{{M}}_{j}} \hat\gamma
_{j}^{(n)}(x|_{j})\to\sum
_{x_{j}\in\mathbb{N}_{*}} \hat\gamma_{j}(x|_{j})\qquad\mbox{as } n\to\infty.
\end{equation}
The a.s. validity of the first part of~(\ref{em2}) for all $j$, $x$, as
well as that of the second part for $k=1$, follow immediately.

In order to get the a.s. validity of the second part of~(\ref{em2}) for
general $k$, we argue as follows. We may suppose by induction that it
holds for $k-1$. We first write the sum in the second part
of~(\ref{em2}) more explicitly as follows:
\begin{equation}
\label{eqem21} \sum_{x_1}\hat\gamma_{1}^{(n)}(x_1)
\cdots\sum_{x_j}\hat\gamma_{j}^{(n)}(x|_{j})
\cdots\sum_{x_k}\hat\gamma_{k}^{(n)}(x|_{k}),
\end{equation}
%
and break each of the $k$ sums (following the strategy of~\cite{kng};
see proof of Proposition 6.3 thereof) in three parts, so that the $j$th
sum is written as
\begin{equation}
\label{eqjs} 
\sum_{x_j}^{(1)}+
\sum_{x_j}^{(2)}+\sum
_{x_j}^{(3)},
\end{equation}
where given $\delta_j\in(0,1)$, the first sum is over $x_j$ such
that $\hat\gamma_{j}(x|_{j})>\delta_j$, the second sum is over
$x|_j$ such that
$M_j^{-1/\alpha_j}<\hat\gamma_{k}(x|_{j})\leq\delta_j$ and the
third sum is over $x|_j$ such that
$\hat\gamma_{k}(x|_{j})\leq M_j^{-1/\alpha_j}$.

It follows from~(\ref{eqgntohg}) 
that
\begin{equation}
\label{eqc1} \sum_{x_1}^{(1)}\hat
\gamma_{1}^{(n)}(x_1)\cdots\sum
_{x_k}^{(1)}\hat\gamma_{k}^{(n)}(x|_{k})
\to \sum_{x_1}^{(1)}\hat
\gamma_{1}(x_1)\cdots\sum_{x_k}^{(1)}
\hat\gamma_{k}(x|_{k})
\end{equation}
as $n\to\infty$ almost surely,
since these are sums over a fixed bounded set of terms. We will show that
\begin{equation}
\label{eqc2} \limsup_{\delta_1,\ldots,\delta_k\to0}\limsup_{n\to\infty}
\sum_{x_1}^{(i_1)}\hat{\gamma}_{1}^{(n)}(x_1)\cdots
\sum_{x_k}^{(i_k)}\hat{\gamma}_{i_k}^{(n)}(x|_{k})=0
\end{equation}
almost surely, for all $(i_1,\ldots,i_k)\in\{1,2,3\}^k\setminus\{
(1,\ldots,1)\}$. Since, again, $\sum_{x_j}^{(1)}$
are sums over a fixed bounded set of terms, and using the induction
hypothesis, it is enough to consider sums
\begin{equation}
\label{eqc3} \sum_{x_1}^{(i_1)}\hat
\gamma_{1}^{(n)}(x_1)\cdots\sum
_{x_k}^{(i_k)}\hat\gamma_{k}^{(n)}(x|_{k})
\end{equation}
with $i_1\in\{2,3\}$.

Let us first consider the case where $i_1=2$ and $i_j\in\{1,2\}$ for
all $j=2,\ldots,k$. It follows from the arguments in the proof of Lemma
6.5 of~\cite{kng} that given $\eta_j>0$ there exists $C_j<\infty$ such
that $\hat\gamma_{j}^{(n)}(x|_j)\leq
C_j[(\hat\gamma_{j}(x|_j))^{1-\eta_j}\vee(\hat\gamma
_{j}(x|_j))^{1+\eta_j}]$, and we replace $\vee$ by $+$, thus obtaining
an upper bound. We then have an upper bound for~(\ref{eqc3}) in terms
of $2^{k-1}$ sums of the form constant times
\begin{equation}
\label{eqc4} \sum_{x_1\dvtx  \hat\gamma_{1}(x_1)\leq\delta_1}\bigl(\hat
\gamma
_{1}(x_1)\bigr)^{1-\eta_1}\sum
_{x_2}\bigl(\hat\gamma_{2}(x|_2)
\bigr)^{1\pm\eta_2} \cdots \sum_{x_k}\bigl(\hat
\gamma_{k}(x|_{k})\bigr)^{1\pm\eta_k}.
\end{equation}
Now, by choosing $\eta_j$ small enough such that
\begin{equation}
\label{eqord} \frac{\alpha_1}{1\pm\eta_1}<\frac{\alpha_2}{1\pm\eta
_2}<\cdots<
\frac{\alpha_k}{1\pm\eta_k}<1,
\end{equation}
one readily checks, for example, by using Campbell's theorem, that for every~$x_1$
\begin{equation}
\label{eqc5} \sum_{x_2}\bigl(\hat
\gamma_{2}(x|_2)\bigr)^{1\pm\eta_2}\cdots\sum
_{x_k}\bigl(\hat\gamma_{k}(x|_{k})
\bigr)^{1\pm\eta_k}
\end{equation}
is an $\frac{\alpha_2}{1\pm\eta_2}$-stable random variable, and
finally that the random variable in~(\ref{eqc4}), which is decreasing
in $\delta_1$, converges in probability to 0 as $\delta_1\to0$. We
conclude it converges almost surely to 0 as $\delta_1\to0$,
and~(\ref{eqc2}) follows for the case where $i_1=2$ and $i_j\in\{1,2\}$
for all $j=2,\ldots,k$.

Let us now analyze the expression in~(\ref{eqc3}) when
$\mathcal{L}:=\{j=1,\ldots,k\dvtx\break   i_j=3\}\ne\varnothing$. It is
argued in~\cite{kng} [see discussion leading to (6.20) in that
reference] that for $j\in\mathcal{L}$, $\hat\gamma_{j}^{(n)}(x|_j)$ is
almost\vspace*{-1pt} surely bounded above by a deterministic constant
times $c_{j}^{(n)}$ for all large enough $n$ uniformly in
$\mathcal{L}$. Let $k'=|\mathcal{L}|$ and let $i'_1<\cdots<i'_{k'}$ be
an enumeration of $\mathcal{L}$, and let $i''_1<\cdots<i''_{k''}$ be an
enumeration of $\{1,\ldots,k\}\setminus\mathcal{L}$, $k''=k-k'$. Then,
arguing as above,~(\ref{eqc3}) may be bounded above by a sum of
$2^{k''}$ terms of the form
\begin{eqnarray}\label{eqc6}
&&\prod_{j=1}^{k'}c_{i'_j}^{(n)}
\sum_{x_{i'_1}=1}^{M_{i'_1}}\cdots\sum
_{x_{i'_{k'}}=1}^{M_{i'_{k'}}} \Biggl\{ \sum
_{x_{i''_1}=1}^{M_{i''_1}} \bigl(\hat\gamma_{i''_1}^{(n)}(x|_{i''_1})
\bigr)^{1\pm\eta_{i''_1}}\cdots
\nonumber\\[-8pt]\\[-8pt]
&& \hspace*{107pt}\sum_{x_{i''_{k''}}=1}^{M_{i''_{k''}}}
\bigl(\hat\gamma_{i''_{k''}}^{(n)}(x|_{i''_{k''}})
\bigr)^{1\pm\eta_{i''_{k''}}} \Biggr\}.\nonumber
\end{eqnarray}
Again, choosing $\eta$'s small enough, we have that the random
variables within braces are i.i.d. $\frac{\alpha_{i''_1}}{1\pm\eta
_{i''_1}}$-stable ones, and since\vspace*{-6pt} the outer sums are over
$\prod_{j=1}^{k'}M_{i'_j}$ terms, and, as one may readily check,
$\prod_{j=1}^{k'}c_{i'_j}^{(n)}(\prod_{j=1}^{k'}M_{i'_j})^{(1\pm\eta
_{i''_1})/\alpha_{i''_1}}$ decays polynomially\vadjust{\goodbreak}
in~$n$ to 0 as $n\to\infty$, by a standard argument, we have that the
expression in~(\ref{eqc6}) decays almost surely to 0 as $n\to\infty$,
and~(\ref{eqhgntohg}) follows for general $k$ by first taking
$n\to\infty$ and then $\delta_1,\ldots,\delta_k\to0$.

It remains to check~(\ref{em4}) and (\ref{em3}) as strong limits for
the $\hat\gamma$ representations of the respective $\gamma$'s. This is
done in much the same way as for checking~(\ref{em2}) above, so we will
be rather sketchy. First note that the expressions in~(\ref{em4}) and
(\ref{em3}) can, after dividing the $M$'s on the denominator inside the
sum, and expanding the resulting products
\begin{equation}
\label{eqprod} \prod_{p = l + 1}^{j-1} \biggl(
\frac{1}{M_{p+1}}+\hat\gamma_{p}^{(n)}(x|_{p})
\biggr),
\end{equation}
be both written as a sum over a fixed number of terms of the form
\begin{equation}
\label{eqem22} \sum_{x_1}\check
\gamma_{1}^{(n)}(x_1)\cdots\sum
_{x_m}\check\gamma_{\ell}^{(n)}(x|_{m}),
\end{equation}
where $1\leq m\leq k$ and $\check\gamma_{j}^{(n)}(x|_j)$ is either
$\hat\gamma_{j}^{(n)}(x|_j)$ or $1/M_{j+1}=c_{j}^{(n)}$ for all
$j=1,\ldots,m$, with the latter case happening for at least one such $j$.

We\vspace*{-3pt} can thus break each sum $\sum_{x_j}$ into three kinds
as above [see~(\ref{eqjs})], with the superscript ``$(3)$'' applying
also to the case where $\check\gamma_{j}^{(n)}(x|_j)=c_{j}^{(n)}$. The
same arguments\vspace*{2pt} used above to estimate the latter cases
of~(\ref{eqc3}) [see the paragraph of~(\ref{eqc6})] apply, since there
is always a sum of the third kind, and the result follows.


\section*{Acknowledgments}
This paper contains results of the PhD thesis of\break  R.~J.~Gava,
supervised by the other authors. The authors thank an anonymous referee
for what can be construed as a careful and thorough reading of an
earlier version of this work, leading to many suggestions and a few
corrections which much improved the presentation. They also thank NUMEC
for hospitality. The work of L.~R.~G. Fontes is part of USP project
MaCLinC and FAPESP project NeuroMat. He would like to thank the CMI,
Universit\'e de Provence, Aix--Marseille I for hospitality and support
during several visits in the last few years where this and related
projects were developed.


%

\printaddresses

\end{document}